\documentclass[11pt]{article}

\usepackage{mathrsfs}

\usepackage[dvips]{epsfig}
\usepackage{graphicx}
\usepackage{amssymb,graphics,amsmath,amsthm,amsopn,amstext,amsfonts, algorithm,algorithmic}
\usepackage{color}

\newcommand{\gap}{\vspace{0.1in}}

\newcommand{\wt}{\widetilde}

\newcommand{\tr}{ \mbox{Tr}}

\newcommand{\Argmin}{\mbox{Argmin}}

\newcommand{\Ical}{\mathcal I}
\newcommand{\Jcal}{\mathcal J}
\newcommand{\Lcal}{\mathcal L}

\newcommand{\diag}{{\rm diag}}
\newcommand{\rank}{\mbox{rank}}

\newcommand{\mycut}[1]{{}}

\newtheorem{lemma}{Lemma}[section] 
\newtheorem{proposition}{Proposition}[section] 
\newtheorem{remark}{Remark}[section]

\begin{document}

\title{A Penalty Decomposition Algorithm with Greedy  Improvement for Mean-Reverting Portfolios with Sparsity and Volatility Constraints}

\author{Ahmad Mousavi\thanks{A. Mousavi  is with the Institute for Mathematics and its Applications, University of Minnesota, Minneapolis, MN 55455, USA. Email: {\tt\small amousavi@umn.edu}.}
\qquad  and  \qquad Jinglai Shen\thanks{J. Shen is with the Department of Mathematics and Statistics, University of Maryland Baltimore County, Baltimore, MD 21250, USA. Email:    {\tt\small shenj@umbc.edu}.}}

\maketitle

\begin{abstract}
Mean-reverting portfolios with few assets, but high variance, are of great interest for investors in financial markets. Such portfolios are straightforwardly profitable because they include a small number of assets whose prices not only oscillate predictably around a long-term mean but also possess enough volatility.  Roughly speaking, sparsity minimizes trading costs, volatility provides arbitrage opportunities, and mean-reversion property equips investors with ideal investment strategies. Finding such favorable portfolios can be formulated as a nonconvex quadratic optimization problem with an additional sparsity constraint. To the best of our knowledge, there is no method for solving this problem and enjoying favorable theoretical properties yet. In this paper, we develop an effective two-stage algorithm for this problem. In the first stage, we apply a tailored penalty decomposition method for finding a stationary point of this nonconvex problem. For a fixed penalty parameter, the block coordinate descent method is utilized to find a stationary point of the associated penalty subproblem. In the second  stage, we improve the result from the first stage  via a greedy scheme that solves restricted nonconvex quadratically constrained quadratic programs (QCQPs).  We show that the optimal value of such a QCQP can be obtained by solving their semidefinite relaxations. Numerical experiments on S\&P 500 are conducted to demonstrate the effectiveness of the proposed algorithm.

\end{abstract}

%
\section{Introduction} \label{sec: intro}

Mean-reversion property plays a significant role in mathematical finance \cite{cuturi2016mean,d2005direct,zhao2019optimal}. In constructing trading portfolios/baskets, this profitable property is traditionally pursued using classical tools in cointegration theory, which often detect a linear combination of assets that are stationary, and consequently, mean-reverting \cite{johansen2009cointegration}. However, such baskets are not practically favorable because they turn out to own many assets that suffer from low volatility. This implies that when the incurred trading costs are not negligible, an investor does not benefit from trading such baskets, that is, sufficient variance provides arbitrage opportunities. Thus, finding mean-reverting portfolios with enough variance has recently attracted much attention; see \cite{zhou2021solving, zhao2019optimal, cuturi2013mean}. Another favorable property for a portfolio is sparsity that helps to minimize trading costs. Consequently, sparse mean-reverting portfolios have been studied \cite{zhang2020sparse, sipos2013optimizing, d2011identifying, fogarasi2013sparse, fogarasi2012simplified, long2018three,sipos2015optimizing}. Sparsity has also shown to be advantageous in numerous applications \cite{shen2018least,mousavi2020survey,mousavi2019quadratic,mousavi2019solution,shen2019exact,mohammadi2020statistical,
mohammadi2021finite,mohammadi2021ultrasound}. While a realistic and practical portfolio should enjoy mean-reversion, volatility, and sparsity properties simultaneously,  there is no method that can effectively solve this problem and capture realistic portfolios, to the best of our knowledge.

Recently, several statistical proxies have been introduced to capture mean-reversion property \cite{cuturi2016mean, d2005direct}.
 In this paper, we consider the following optimization problem that aims to minimize the predictability notion introduced by Box and Tiao \cite{box1977canonical} while ensuring sparsity and volatility:
\begin{equation} \label{pr: p}
(P): \qquad \min_{x \in \mathbb{R}^N} \ x^T M x \qquad \mbox{subject to}  \qquad  x^T A x\ge \phi, \quad  x^T x=1, \quad \text{and}\quad \|x\|_0\le k,
\end{equation}
where  $M, A\in \mathbb R^{N\times N}$ are symmetric and positive definite, $\phi$ is a positive number, $\|\cdot\|_0$ denotes the number of nonzero entries of a vector, and $k\in \mathbb N$ with $k\ll N$.
As mentioned before, the only method proposed for solving this problem based on the following semidefinite program (SDP) relaxation \cite{cuturi2016mean}:
\begin{equation} \label{pr: sdp-of-p}
\min_{Y\in \mathcal S^N}  \quad
\tr(MY)+\rho \|Y\|_1         \qquad
\textrm{subject to} \qquad
 \tr(AY)\ge \phi, \quad  \tr(Y)=1, \quad  \text{and} \quad Y\succeq 0,
\end{equation}
where $\rho>0$ is a penalty parameter, and $\|Y\|_1:=\sum_{i,j} |Y_{ij}|$. The $\ell_1$ norm  promotes  the sparsity of the decision variable $Y$. After solving the convex SDP (\ref{pr: sdp-of-p}) and obtaining a solution $Y^*$, the authors in \cite{d2005direct}  apply sparse PCA to $Y^*$ for recovering a solution $y^*$ of (\ref{pr: p}). The major drawbacks of this method are as follows. A solution of
(\ref{pr: sdp-of-p}) may not be of low-rank in general. This hinders one from obtaining a rank-one solution of (\ref{pr: sdp-of-p}) to recover a solution for the original problem (\ref{pr: p}). Hence, in general, the SDP relaxation formulation (\ref{pr: sdp-of-p}) may give rise to a solution quite different from that of (\ref{pr: p}).
%
%
In \cite{cuturi2016mean}, the authors suggest to solve the following sparse PCA if a rank-one solution is available:
\begin{equation} \label{pr: sparse-pca}
(s-PCA): \qquad \min_{y\in \mathbb{R}^N} \ y^THy \qquad \mbox{subject to}  \qquad  y^Ty=1, \quad \text{and}\quad \|y\|_0\le k.
\end{equation}
We emphasize again that obtaining a rank-one solution of (\ref{pr: sdp-of-p}) is not guaranteed even though it could be the case in practice. Further, no theoretical results are established for the qualitative properties of an output of the sparse PCA with respect to (\ref{pr: p}).
%
%
In summary, the SDP relaxation formulation (\ref{pr: sdp-of-p}) does not necessarily yield a solution, or a related solution such as a stationary point, to (\ref{pr: p}) in general.

In view of the drawbacks of (\ref{pr: sdp-of-p}), we propose an effective two-stage algorithm for solving (\ref{pr: p}) that not only attains favorable theoretical properties but also achieves satisfactory numerical performance. In the first stage, we apply a tailored penalty decomposition method that finds a stationary point of (\ref{pr: p}). When applying this method, each penalty subproblem is nonconvex but decomposed such that we apply block coordinate minimization to find a stationary point of a penalty subproblem. The restricted subproblems for the block coordinate minimization are tractable since one subproblem attains a closed-form solution, and the other subproblem can be handled via its SDP relaxation that provably achieves the exact optimal value and further finds a rank-one solution corresponding to a solution of the original subproblem efficiently. In the second stage, we propose a greedy scheme that starts from the stationary point obtained from the first stage and then improves upon it via solving sparsity-sized semidefinite programs. This greedy scheme stops once an index set cannot be further improved. We show that the SDP used in this step indeed achieves the exact optimal value as its original nonconvex QCQP. We carry out numerical testes and compare the proposed algorithm  with the method in \cite{cuturi2016mean} on the S\&P 500 assets. The numerical results show that our algorithm outperforms the latter method.

The rest of the paper is organized as follows. In Section \ref{sec: greedy_penalty_decomposition_method}, we discuss optimality conditions and give an overview of the proposed two-stage algorithm. Section \ref{sec:finding_stationary_point} studies the first stage in detail and establishes theoretical properties. In Section \ref{sec: greedy_stage}, a greedy scheme is proposed to improve a stationary point obtained from stage one. Numerical experiments and results are shown and discussed in Section \ref{sec: numerical_experiments}.


{\it Notation}.
For a set $S$, we denote its complement as either $S^c$ or $\bar S$ and its cardinality as $|S|$. For a natural number $N$, let   $[N]$ be$\{1,2,\dots,N\}$. For $S=\{i_1,i_2\dots, i_{|S|}\} \subseteq [N]$,  let $x_S\in \mathbb R^{|S|}$ is the coordinate projection of $x$ with respect to indices in $S$. For vectors $x_i\in \mathbb R^{N_i}$ with $i\in [N]$, we denote their column concatenation vector in $\mathbb R^{\sum_{i=1}^N N_i}$ using Matlab notation  as $[x_1;x_2;\dots;x_N]$.
Consequently, a $k$-sparse vector $x\in \mathbb R^N$ supported on $S\subseteq [N]$ with $|S|=k$ is written as $x=[x_S;0]$.
Similarly, if  $N_i=m$ for all $i\in [N]$, we denote the row concatenation matrix  of vectors  $x_i\in \mathbb R^{m}$  as $[x_1 \, x_2\, \dots\, x_N]\in \mathbb R^{m\times N}$. These notations are also used for matrices of suitable sizes.
The space of  $N\times N$ symmetric matrices is denoted by $\mathcal S^N$. We  write  $A\succ 0$ and  $A\succeq 0$ for positive  definiteness and semi-definiteness of $A$, respectively.
The smallest and largest eigenvalues of a symmetric matrix $A$ are denoted by $\lambda_{\min}(A)$ and $\lambda_{\max}(A)$.
%
%
The identity matrix of size $m$ is denoted by $I_m$. A diagonal matrix $D$ with diagonal entries $d_1,d_2,\dots,d_m$ is denoted by $D=\diag(d_1,d_2,\dots, d_m)$.  The trace of a square matrix $A$ is denoted by $\tr(A)$.

%

%

%
\section{Penalty Decomposition Algorithm with  Greedy  Improvements} \label{sec: greedy_penalty_decomposition_method}

%
\subsection{Optimality Condition of the Mean Reverting Problem}

%
%

The problem $(P)$ in (\ref{pr: p}) attains an optimal solution if it is feasible.
Let $x^* \in \mathbb R^N$ be a local minimizer of the problem $(P)$, and $\Ical$ be an index subset of $\{1, \ldots, N\}$ such that $|\Ical| = k$, and $x^*_{\Ical^c}=0$. Note that $\Ical$ may not be the support of $x^*$.
It is easy to show that $x^*$ is also a local minimizer of the problem
\[
  (P'): \quad  \min_{x \in \mathbb R^N} x^T M x, \qquad \mbox{ subject to } \quad x^T A x \ge \phi, \ \ x^T x =1, \quad \mbox{ and } \quad x_{\Ical^c}=0,
\]
Hence, the Robinson's condition for a local minimizer $x^*$ of $(P')$ with the index set $\Ical^c$ with $x^*_{\Ical^c}=0$ is \cite[Eqn.(3.11)]{ruszczynski2011nonlinear}
%
%
\begin{itemize}
   \item [(i)] $(x^*)^T A x^*> \phi$.
     $
        \left\{  \begin{bmatrix} -2 d^T  A x^*  - v \\ 2 d^T x^* \\ d_{\Ical^c} \end{bmatrix} \  \Big | \ \  d \in \mathbb R^{N}, \ \ v \in \mathbb R \right\} = \mathbb R \times \mathbb R \times \mathbb R^{|\Ical^c|};
     $
   \item [(ii)]  $(x^*)^T A x^*= \phi$.
     $
        \left\{  \begin{bmatrix} -2 d^T  A x^*- v \\ 2 d^T x^* \\ d_{\Ical^c} \end{bmatrix} \  \Big | \ \  d \in \mathbb R^{N}, \ \ v \in \mathbb R_- \right\} = \mathbb R \times \mathbb R \times \mathbb R^{|\Ical^c|}.
     $
\end{itemize}
Since $\| x^*\|_2=1$ and $x^*_{\Ical^c}=0$, we have $x^*_\Ical \ne 0$. Hence, the Robinson's condition for case (i) always holds. For case (ii), the necessary and sufficient condition for Robinson's condition is given below.
%
%
\begin{lemma} \label{lem:Robison_condition_ii}
  Suppose a local minimizer $x^*$ of $(P)$ is such that $(x^*)^T A x^*= \phi$ and $x^*_{\Ical^c}=0$. Then Robinson's condition holds if and only if $\{ (A x^*)_\Ical, x^*_\Ical\}$ is linearly independent.
\end{lemma}

\begin{proof}
 ``If'': suppose $\{ (A x^*)_\Ical, x^*_\Ical\}$ is linearly independent. Then $\big\{ (d^T_{\Ical}  (A x^*)_\Ical, d^T_{\Ical} x^*_\Ical) \in \mathbb R^2 \, | \, d_{\Ical} \in \mathbb R^{|\Ical|} \big\} = \mathbb R^2$. Clearly, this yields Robinson's condition for case (ii).
%
%
 To see ``Only If'', suppose Robinson's condition holds but $\{ (A x^*)_\Ical, x^*_\Ical\}$ is linearly dependent. Since $x^*_{\Ical^c} = 0$, $x^*_\Ical \ne 0$, and $A$ is PD, we have $(A x^*)_\Ical = A_{\Ical \Ical} \cdot x^*_\Ical = \nu  \cdot x^*_\Ical$ for some positive constant $\nu>0$. Let $d_{\Ical^c}=0$. By Robinson's condition, we must have that the set $S:=\{ ( -2 z^T  (A x^*)_\Ical - v, 2 z^T x^*_\Ical) \, | \, z \in \mathbb R^{|\Ical|}, v \le 0 \} = \mathbb R^2$. On the other hand, for any $z$ with $z^T x^*_\Ical \le 0$ and $v \le 0$, we have $-2 z^T  (A x^*)_\Ical - v = -2 \nu \cdot (z^T x^*_\Ical) - v \ge 0$. Therefore, $S$ does not contain the interior of $\mathbb R_-\times \mathbb R_-$, yielding a contradiction to $S=\mathbb R^2$. This shows that $\{ (A x^*)_\Ical, x^*_\Ical\}$ must be linearly independent.
\end{proof}

Under the Robinson's condition, the first-order optimality condition (i.e., the KKT condition) for a local minimizer $x^*$ of $(P)$ (or $(P')$) is: there exist $\lambda, \mu \in \mathbb R$ and $w = (w_\Ical, w_{\Ical^c})\in \mathbb R^N$  such that
\begin{equation} \label{eqn:KKT_conditions}
   M x^* - \lambda  A x^* +  \mu x^* + w = 0, \quad 0 \le \lambda \perp (x^*)^T A x^* - \phi \ge 0, \quad \|x^* \|_2 =1, \ \ w_\Ical=0, \ \ x^*_{\Ical^c} = 0.
\end{equation}

%
\subsection{Overview of Penalty Decomposition Algorithm with  Greedy Improvement}

This paper develops a penalty decomposition scheme \cite{lu2013sparse} along with a greedy algorithm to solve the mean-reverting problem $(P)$ in (\ref{pr: p}). We provide an overview of the proposed algorithm in this subsection.

The entire algorithm consists of two stages. In the first stage, the penalty decomposition scheme \cite{lu2013sparse} is exploited to obtain a stationary point of the problem $(P)$.
The penalty decomposition scheme solves a sequence of simpler penalty subproblems using the block coordinate decent (BCD) method at each step. Under some mild assumptions, a subsequence of the iterations yielded by the penalty decomposition scheme converges to a stationary point of $(P)$.
In the second stage, a greedy algorithm is applied to improve the result found from stage one. This greedy algorithm updates the associated support set of the current iterates by solving a sequence of restricted nonconvex QCQPs. Further, such a restricted nonconvex QCQP can be efficiently solved via SDP relaxation which achieves  the exact optimal value of the nonconvex QCQP.
The greedy algorithm stops once the support set cannot be improved.
The entire algorithm is summarized in Algorithm \ref{algo: ystar-p}.
\begin{algorithm}[H]
\caption{Penalty Decomposition Method with Greedy Improvement for Solving Problem (\ref{pr: p})}
\begin{algorithmic}[1]
\label{algo: ystar-p}
\STATE Inputs: $N\times N$ positive definite matrices $M$ and $A$,   $\phi>0$, and $k\in \mathbb N$.
\STATE \fbox{Stage 1:} run penalty decomposition scheme (cf. Algorithm \ref{algo:PD_scheme}).
%
%
Take its output as the index set $\Lcal$ with $|\Lcal|=k$ associated with a stationary point $x^*$.

\STATE \fbox{Stage 2:} run a greedy algorithm (cf. Algorithm \ref{algo:greedy_scheme}) using $\Lcal$ 

\end{algorithmic}
\end{algorithm}

%
\section{Stage One: Penalty Decomposition Scheme} \label{sec:finding_stationary_point}

We show how to find a stationary point of the nonconvex problem (\ref{pr: p}) via a penalty decomposition scheme \cite{lu2013sparse}. By introducing a new variable $y \in \mathbb R^N$, we reformulate (\ref{pr: p}) as:
\begin{equation}\label{pr: pxy} 
\min_{(x,y)\in \mathbb{R}^N \times \mathbb R^N} \ x^TMx\qquad \mbox{subject to}  \qquad  x^TAx\ge \phi, \quad y^Ty=1,  \quad  \|y\|_0\le k, \quad \text{and} \quad  x-y=0.
\end{equation}
Define the sets
\begin{equation*} 
    \mathcal X:=\{x\in \mathbb{R}^N \, | \, x^TAx\ge \phi\} \qquad \mathrm{and} \qquad \mathcal Y:=\{y\in \mathbb{R}^N \, | \,  y^T y=1 \quad \text{and} \quad \|y\|_0\le k\},
\end{equation*}
and the quadratic penalty function $q_{\rho}(x,y):= x^ TMx +\rho \|x-y\|_2^2$ for a given positive constant $\rho$. Consider the following problem:
\begin{equation} \label{eqn:P_x_y}
  (P_{x, y}): \quad \min_{(x, y) \in \mathbb R^N \times \mathbb R^N} q_{\rho}(x, y), \qquad \mbox{ subject to }  \qquad x \in \mathcal X, \quad y \in \mathcal Y.
\end{equation}
Clearly, its solution exists. Given a (local) minimizer $(x_*, y_*)$ of $ (P_{x, y})$, define the index set $\mathcal L \subseteq\{1, \ldots, N \}$ such that $|\Lcal|=k$ and $y^*_{\Lcal^c}=0$.
It is easy to show that $(x_*, y_*)$ is a local minimizer of the problem $(P_{x, y})$ when $\mathcal Y$ is replaced by $\mathcal Y' := \{ y \in \mathbb R^N \, | \, \| y \|_2=1, \ y_{\Lcal^c}=0 \}$. We denote the latter problem (with $\mathcal Y'$ instead of $\mathcal Y$) by $(P'_{x, y})$. It is easy to show that the Robinson's condition holds for $(P'_{x, y})$ at any feasible $(x, y) \in \mathbb R^N\times \mathbb R^N$. Hence, the KKT condition for a local minimizer $(x_*, y_*)$ of $(P_{x, y})$ (or equivalently $(P'_{x, y})$) is: there exists $\lambda, \mu \in \mathbb R$ such that
\begin{equation} \label{eqn:KKT_P_x_y}
   M x_* - \lambda A x_* + \rho (x_* - y_*)= 0, \ (\rho+\mu) (y_*)_{\Lcal} = \rho (x_*)_{\Lcal}, \ \ 0 \le \lambda \perp x_*^T A x_* - \phi \ge 0, \  \|y_* \|_2 =1, \  (y_*)_{\Lcal^c} = 0.
\end{equation}

The paper \cite{lu2013sparse} develops a block coordinate decent (BCD) scheme given in Algorithm~\ref{algo:BCD_scheme} to compute a saddle point of $(P_{x, y})$, i.e., $(x_*, y_*) \in \mathcal X \times \mathcal Y$ such that
\begin{equation} \label{eqn:saddle_point}
  x_* \in \mbox{Argmin}_{x\in \mathcal X} \ q_\rho(x, y_*), \qquad  y_* \in \mbox{Argmin}_{y\in \mathcal Y} \ q_\rho(x_*, y).
\end{equation}
We discuss the two subproblems in the above formulation as follows.

Given $y \in \mathbb R^N$ and $\rho>0$, consider the problem
\begin{equation} \label{pr: px}
 (P_x): \qquad \min_{ x \in \mathbb R^N} x^T M x + \rho \| x - y \|^2_2, \qquad \mbox{ subject to }  \qquad x^T A x \ge \phi.
\end{equation}
Clearly, its optimal solution exists and the constraint qualification holds. Hence, the KKT condition for a (local) minimizer $x_*$ of $(P_x)$ is given by: there exists $\lambda \in \mathbb R$ such that
\[
   M x_* + \rho(x_* - y) -  \lambda A x_* =0, \qquad 0 \le \lambda \perp  x^T_* A x_* - \phi \ge 0.
\]
The following result shows that for any given $y$, an optimal solution $x_*$ to $(P_x)$ is bounded.

\begin{lemma} \label{lem:opt_bound_P_x}
  Given $\rho>0$ and $y \in \mathbb R^N$, a local minimizer $x_*$ of $(P_x)$ satisfies $$\| x_*\|_2 \le \max\left( \sqrt{\frac{\phi}{\lambda_{\min}(A)} }, \ \|y \|_2 \right).$$
\end{lemma}

\begin{proof}
 A local minimizer $x_*$ satisfies either $x^T_* A x_* = \phi$ or $x^T_* A x_*>\phi$. For the former case, we have $\|x\|^2_2 \le \phi/ \lambda_{\min}(A)$. For the latter, the multiplier $\lambda =0$ such that $M x_* + \rho(x_* - y)=0$. Hence, $(M + \rho I) x_* = \rho y$, leading to $x_* = (M+ \rho I)^{-1} \rho y$. This shows that $\| x^* \| \le \| (M+ \rho I)^{-1}\|_2  \cdot \rho \cdot \|y \|_2 \le \frac{\rho}{ \lambda_{\min}(M) + \rho} \| y \|_2 \le \| y \|_2$.
\end{proof}

Given $0 \ne x \in \mathbb R^N$, consider the problem
\begin{equation} \label{pr: py}
 (P_y): \qquad \min_{ y \in \mathbb R^N} \| y - x \|^2_2, \qquad \mbox{ subject to }  \qquad y^T y =1, \quad \mbox{ and } \quad \| y \|_0 \le k.
\end{equation}
To solve this problem in a closed form, let $\Jcal(x, k) \subseteq \{1, \ldots, N\}$ be the index set corresponding to the first $k$ largest elements of $x$ in absolute values.

\begin{lemma} \label{lem:sol_P_y}
 Given $ 0 \ne x \in \mathbb R^N$, let $\Jcal:=\Jcal(x, k)$. Then an optimal solution to $(P_y)$ is given by $y^* = (y^*_\Jcal, y^*_{\Jcal^c})$, where $y^*_\Jcal = \frac{ x_\Jcal}{ \|  x_\Jcal \|_2 }$, and $y^*_{\Jcal^c}=0$.
\end{lemma}

\begin{proof}
 Note that for any $y$ satisfying $\| y \|_0 \le k$, it can be written $y_\Ical =0$ for some index set $\Ical \subseteq \{1, \ldots, N \}$ with $|\Ical| =N-k$. Hence, for any index set $\Ical$ with $|\Ical| =N-k$, $(P_y)$ can be written as $\min_{ y \in \mathbb R^N} \| y - x \|^2_2$ subject to $y^T y =1$ and $y_{\Ical}=0$, which is equivalent to $\min_{z \in \mathbb R^{|\Ical^c|} } \| z - x_{\Ical^c} \|^2_2$ subject to $z^T z =1$. Clearly, constraint qualification holds and its KKT condition for a local minimizer $z_*$ is: $z_* - x_{\Ical^c} + \mu z_* =0$ and $\|z_*\|_2=1$ for some $\mu \in \mathbb R$. This shows that $(1+\mu) z_* = x_{\Ical^c}$. Without loss of generality, we assume that $x_{\Ical^c} \ne 0$ (otherwise, $z_*$ is arbitrary on the sphere $\|z \|_2=1$). Then we must have $1+\mu \ne 0$ such that $z_* = \frac{1}{1+\mu} x_{\Ical^c}$ or equivalently $z_*$ is parallel to $x_{\Ical^c}$. Hence, it is easy to show that $z_* = \frac{x_{\Ical^c} }{\| x_{\Ical^c} \|_2}$ for any index $\Ical$ specified above. Finally, for any index $\Ical$ specified above, the optimal value is given by
 $ \|x_{\Ical^c} -  \frac{x_{\Ical^c} }{\| x_{\Ical^c} \|_2} \|^2_2 + \|x_{\Ical}\|^2_2 = (\|x_{\Ical^c} \|_2 - 1)^2 + \|x_{\Ical}\|^2_2 = \|x \|^2_2 +1 - 2 \|x_{\Ical^c}\|_2$. Consequently, the minimal value of $(P_y)$ is achieved when $\|x_{\Ical^c}\|_2$ is maximal or equivalently when $\Ical^c = \Jcal(x, k)=\Jcal$. Therefore, a minimizer $y^*$ satisfies $y^*_\Jcal = \frac{ x_\Jcal}{ \|  x_\Jcal \|_2 }$, and $y^*_{\Jcal^c}=0$.
\end{proof}

 Returning to the problem given by (\ref{eqn:saddle_point}), we see that $x_*$ is a solution to $(P_x)$ when $y=y_*$, and $y_*$ is a solution to  $(P_y)$ when $x=x_*$. (Note that such a saddle point exists.) The first order necessary conditions for a saddle point $(x_*, y_*) $ is: (i) there exists $\lambda \in \mathbb R$ such that
\[
   M x_* + \rho(x_* - y_*) -  \lambda A x_* =0, \qquad 0 \le \lambda \perp  x^T_* A x_* - \phi \ge 0,
\]
and (ii) $y_* = \mathcal T(x_*)$, where $\mathcal T$ is the operator for the closed form solution to $(P_y)$. Particularly, let $\Lcal=\Jcal(x_*, k)$. Then $(y_*)_{\Lcal} = \frac{ (x_*)_{\Lcal} }{ \| (x_*)_{\Lcal}  \|_2}$.
 Since $y_*$ satisfies $\|y_* \|_2=1$, it follows from Lemma~\ref{lem:opt_bound_P_x} that $\| x_* \|_2 \le \max\left( \sqrt{\frac{\phi}{\lambda_{\min}(A)} }, \ 1 \right)$ for any $\rho>0$. It is easy to verify that a saddle point must be a stationary point of $(P_{x, y})$ satisfying the first order optimality conditions in (\ref{eqn:KKT_P_x_y}) with $\mu = \rho(\| (x_*)_{\Lcal}  \|_2 -1 )$. 

\begin{algorithm}
\caption{Block Coordinate Decent Scheme for $(P_{x, y})$ in (\ref{eqn:P_x_y})}
\begin{algorithmic}[1]
\label{algo:BCD_scheme}
\STATE Initialization with a given constant $\rho>0$, $s=0$ and $y^s \in \mathcal Y$
\REPEAT
\STATE Compute $x^{s+1} \in \mbox{Argmin}_{x\in \mathcal X} \ q_\rho(x, y^s)$

\STATE Compute $y^{s+1} \in \mbox{Argmin}_{y \in \mathcal Y} \ q_\rho(x^{s+1}, y)$ using the operator $\mathcal T$
%
%
\STATE $s \leftarrow s+1$
\UNTIL{Stopping criterion is met}
%
%
\end{algorithmic}
\end{algorithm}

The following lemma shows that the sequence $( q_\rho(x^s, y^s))$ is either strictly decreasing or reaches an equality at a finite step, which yields a saddle point.

\begin{lemma} \label{lem:P_x_y_saddle_pt}
  Given a constant $\rho>0$, let $\big (x^s, y^s) ) $ be a numerical sequence generated by the BCD scheme in Algorithm~\ref{algo:BCD_scheme}. Then $( q_\rho(x^s, y^s))$ is non-increasing. Further, if $q_\rho(x^r, y^r) =  q_\rho(x^{r+1}, y^{r+1})$ for some $r\in \mathbb N$, then $(x^r, y^r)$ is a saddle point of $(P_{x, y})$.
\end{lemma}

\begin{proof}
 It follows from the proof for \cite[Theorem 4.2]{lu2013sparse} that $q_\rho(x^{s+1}, y^{s+1}) \le q_\rho(x^{s+1}, y^s) \le q_\rho(x^s, y^s)$ for all $s$. Hence, $( q_\rho(x^s, y^s))$ is non-increasing. Now suppose $q_\rho(x^r, y^r) =  q_\rho(x^{r+1}, y^{r+1})$ for some $r\in \mathbb N$. Then by the above result, $ q_\rho(x^{r+1}, y^r) = q_\rho(x^r, y^r)$. Furthermore, since $x^{r+1} \in \mbox{Argmin}_{x\in \mathcal X} \ q_\rho(x, y^r)$, we have $q_\rho (x^{r+1}, y^r) = \min_{x \in \mathcal X} q_\rho(x, y^r)$. In view of $ q_\rho(x^{r+1}, y^r) = q_\rho(x^r, y^r)$ and $x^r \in \mathcal X$, we have $q_\rho(x^r, y^r)=\min_{x \in \mathcal X} q_\rho(x, y^r)$ such that $x^r \in \mbox{Argmin}_{x\in \mathcal X} \ q_\rho(x, y^r)$. Further, $y^r \in \mathcal Y$ satisfies $y^r \in \mbox{Argmin}_{y \in \mathcal Y} \ q_\rho(x^{r}, y)$. This shows that $(x^r, y^r)$ is a saddle point of $(P_{x,y})$.
\end{proof}

We comment on two computational issues when running Algorithm~\ref{algo:BCD_scheme} to solve $(P_{x, y})$.

\begin{remark} \rm
In view of Lemma~\ref{lem:P_x_y_saddle_pt}, a stopping criterion for Algorithm~\ref{algo:BCD_scheme} can be based on the relative error of $q_{\rho}(x^s, y^s)$, namely,
$
\displaystyle    \frac{ q_{\rho}(x^s, y^s) -  q_{\rho}(x^{s+1}, y^{s+1}) }{ q_{\rho}(x^s, y^s)} \le \varepsilon
$
for a sufficiently small $\varepsilon>0$.
Another stopping criterion suggested in \cite{lu2013sparse} is: for a sufficiently small $\varepsilon>0$,
\[
\max
\left\{
\frac{\|x_s-x_{s-1}\|_\infty}{\max \left(\|x_s\|_\infty,1 \right)}, \
\frac{\|y_s-y_{s-1}\|_\infty}{\max \left(\|y_s\|_\infty,1 \right)}
\right\}
\le \varepsilon.
\]
\end{remark}

\begin{remark} \rm
The problem $(P_{y})$ given in Line 4 of Algorithm~\ref{algo:BCD_scheme} has a closed form solution defined by the operator $\mathcal T$ shown in Lemma~\ref{lem:sol_P_y}.
To solve the problem $(P_{x})$ in Line 3 for a given $y$, we exploit SDP  relaxation.
  Note that the   SDP  relaxation of a quadratic program with exactly one quadratic constraint obtains the same optimal value  provided that it is strictly feasible \cite{boyd2004convex}. It is easy to see (\ref{pr: px}) is strictly feasible because  $A\succ 0$. Hence, we can  find the optimal value of this nonconvex problem exactly via solving its convex  SDP  relaxation below:
\begin{equation*} 
\begin{aligned}
\min_{X\in \mathcal S^{N+1}}  \quad &
\tr\left(
\begin{bmatrix} \rho \|y\|_2^2 & -\rho y^T\\ -\rho y &  M+\rho I_N  \end{bmatrix} X
\right)
\\
\textrm{subject to} \quad
 &
 \tr\left(
 \begin{bmatrix} -\phi & 0\\ 0& A \end{bmatrix}
 X\right)\ge 0, \quad X_{11}=1,
\quad \text{and} \quad
X\succeq 0.               \\
\end{aligned}
\end{equation*}
Its dual problem is given by
\begin{equation*} 
\begin{aligned}
\max_{w_1 \in \mathbb R, w_2\in \mathbb R, Z\in \mathcal S^{N+1}}  \quad & w_2
\\
\textrm{subject to} \quad
 &
 w_1 \begin{bmatrix} -\phi & 0\\ 0& A \end{bmatrix}
 +w_2
 \begin{bmatrix} 1 & 0\\ 0& 0\end{bmatrix}+Z=
 \begin{bmatrix} \rho \|y\|_2^2 & -\rho y^T\\ -\rho y &  M+\rho I_N  \end{bmatrix},
 \quad w_1\ge 0,
\quad \text{and} \quad
Z\succeq 0.               \\
\end{aligned}
\end{equation*}
We first show that both problems are strictly feasible. Clearly, $X=\begin{bmatrix} 1 & 0\\ 0& \gamma I_N \end{bmatrix}$ with $\gamma>1/\tr(A)$ is a strictly feasible point of the primal problem (recall that $A\succ 0$ and thus $\gamma>0$). To see the dual problem is strictly feasible, it is enough to show that there exist a positive $w_1$ and an arbitrary $w_2$ such that  $Z=
 \begin{bmatrix} \rho \|y\|_2^2+w_1\phi-w_2 & -\rho y^T\\ -\rho y &  M+\rho I_N-w_1A  \end{bmatrix}\succ 0$.
 This block matrix is positive definite if and only if (i) $M+\rho I_N-w_1A\succ 0$ and (ii) $\rho \|y\|_2^2+w_1\phi-w_2 -\rho^2y^T(M+\rho I_N-w_1A)^{-1}y>0$. To guarantee the inequality (i), since $M\succ 0 $ and $\rho>0$, it is enough to choose $w_1=\epsilon>0$ small enough such that $\lambda_{\min}(M)+\rho >\epsilon \lambda_{\max}(A)$. The inequality (ii) can be easily guaranteed by choosing $w_2$ appropriately.
Since the primal and dual problems are strictly feasible, both problems obtain their solutions with the same optimal value. Let $X^*$ and $(w_1^*, w_2^*, Z^*)$ be the optimal solutions. If $X^*$ is rank-one, we trivially have the solution of (\ref{pr: px}). Otherwise, by applying the rank-one decomposition procedure in  \cite[Lemma 2.2]{ye2003new}, we get
$X^*=\sum_{i=1}^r u_iu_i^T$ with $r=\rank(X^*)$, $0 \ne u_i \in \mathbb R^{N+1}$, for all $i\in [r]$ such that
\begin{equation*}
    ru_i^T\begin{bmatrix} -\phi & 0\\ 0& A \end{bmatrix}u_i=\tr\left(
    \begin{bmatrix} -\phi & 0\\ 0& A \end{bmatrix}X^*
    \right)
    \ge 0, \quad \forall \, i\in [r].
\end{equation*}
Since $X^*_{11}=1$, there exists $j\in [r]$ such that $u_j=[\alpha;u] \in \mathbb R^{N+1}$ such that $\alpha\ne 0$. Further, the KKT conditions imply that $0=\tr\left(X^*Z^*\right)=\tr\left(\sum_{i=1}^ru_iu_i^TZ^*\right)=\sum_{i=1}^r\tr\left(u^T_iZ^*u_i\right)$ and since $Z^*\succeq 0$, we have $u_j^TZ^*u_j=0$. Thus, $u_ju_j^T$ and $(w_1^*, w_2^*, Z^*)$  satisfy the KKT conditions and consequently, $u/\alpha$ yields a solution to (\ref{pr: px}). For more details, see \cite[Section 2.1]{ye2003new}.
\end{remark}



The penalty decomposition scheme developed in \cite{lu2013sparse} is applied to compute a stationary point of $(P)$ (or $P'$); see Algorithm~\ref{algo:PD_scheme}. The inner loop of  Algorithm~\ref{algo:PD_scheme}, which is as same as  Algorithm~\ref{algo:BCD_scheme}, intends to solve $(P_{x, y})$ for a given $\rho_j>0$.

%
%

\begin{algorithm}
\caption{Penalty Decomposition Scheme for $(P)$ in (\ref{pr: p}) }
\begin{algorithmic}[1]
\label{algo:PD_scheme}
\STATE Initialization with constants $r>1$ and $\rho_0>0$, set $j=0$, and choose $y^{0, 0}\in \mathcal Y$
\REPEAT

 \STATE Set $s=0$

 \REPEAT

   \STATE Compute $x^{j, s+1} \in \mbox{Argmin}_{x\in \mathcal X} \ q_{\rho_j}(x, y^{j, s})$

    \STATE Compute $y^{j, s+1} \in \mbox{Argmin}_{y \in \mathcal Y} \ q_{\rho_j}(x^{j, s+1}, y)$ using the operator $\mathcal T$

  \STATE $s \leftarrow s+1$

 \UNTIL{Stopping criterion is met}

%
%
\STATE  $\rho_{j+1} = \rho_j \cdot r$, and $y^{j+1, 0} = y^{j, s}$
\STATE $j \leftarrow j+1$
\UNTIL{Stopping criterion is met}
%
%
\end{algorithmic}
\end{algorithm}

In what follows, we show that a numerical sequence of Algorithm~\ref{algo:PD_scheme} attains a convergent subsequence with the limit $(x^*, y^*)$ and that under a mild assumption on $x^*$, $x^*$ is a KKT point of a local minimizer of $(P)$ in (\ref{pr: p}) or $(P')$ for a suitable index set $\Lcal$ with $x^*_\Lcal=0$. The proof this result mostly follows from that of \cite[Theorem 4.3]{lu2013sparse}. To be self contained, we present it here with emphasis on some (minor) differences.

%
%

\begin{proposition}
 Let  $\big ((x^j, y^j) \big)$  be a sequence generated by Algorithm~\ref{algo:PD_scheme}. The following hold:
 \begin{itemize}
   \item [(i)] $\big( ((x^j, y^j) \big)$ has a convergent subsequence whose limit is given by $(x^*, y^*)$ satisfying $x^*=y^*$, and there exists an index subset $\Lcal$ with $|\Lcal|=k$ such that $x^*_{\Lcal^c}=0$.
   \item [(ii)] Suppose the Robinson condition given before Lemma~\ref{lem:Robison_condition_ii} holds at $x^*$ with the index subset $\Lcal$ indicated above. Then $x^*$ is a KKT point satisfying (\ref{eqn:KKT_conditions}) for $(P')$ with $\Ical=\Lcal$.
 \end{itemize}
\end{proposition}

\begin{proof}
(i) For each $j$, $(x^j, y^j)$ is a saddle point of $(P_{x, y})$ in (\ref{eqn:P_x_y}) with $\rho=\rho_j>0$, namely,
 $x^j \in \mbox{Argmin}_{x\in \mathcal X} \ q_{\rho_j}(x, y^j)$ and  $y^j \in \mbox{Argmin}_{y\in \mathcal Y} \ q_{\rho_j}(x^j, y)$. By the definition of $\mathcal Y$, $\| y^j \|_2 = 1$ for all $j$. Further, it follows from Lemma~\ref{lem:opt_bound_P_x} that
  $\| x^j \|_2 \le \max\left( \sqrt{\frac{\phi}{\lambda_{\min}(A)} }, \ 1 \right)$ for any $\rho_j>0$. Thus  $\big( ((x^j, y^j) \big)$ is bounded and thus attains a convergent subsequence whose limit is denoted by $(x^*, y^*)$. By the similar argument in \cite[Theorem 4.3]{lu2013sparse}, we see that $x^*=y^*$.
   Let $\Ical^j$ be an index subset of $[N]$ with $|\Ical^j|=k$ and $y^j_{(\Ical^j)^c}=0$ for each $j$. Then there exist an index subset $\Lcal$ with $|\Lcal|=k$ and a subsequence $\big( (x^{j_\ell}, y^{j_\ell}) \big)$ of the above convergent subsequence such that $\Ical^{j_\ell} = \Lcal$ for all large $j_\ell$. Therefore, $x^*_{\Lcal^c}=0$.

(ii)  Using the above mentioned subsequence, we assume without loss of generality that $\big( (x^{j_\ell}, y^{j_\ell} ) \big)$ converges to $(x^*, y^*)$ and $y^{j_\ell}_{\Lcal^c}=0$ for all $i_\ell$. Since each $(x^{j_\ell}, y^{j_\ell})$ is a saddle point of $(P_{x, y})$ with $\rho=\rho_{j_\ell}>0$, it satisfies
 \[
   M x^{j_\ell} + \rho_{j_\ell} (x^{j_\ell} - y^{j_\ell}) -  \lambda_{j_\ell} A x^{j_\ell} =0, \qquad 0 \le \lambda_{j_\ell} \perp  (x^{j_\ell})^T A x^{j_\ell} - \phi \ge 0,
 \]
and $y^{j_\ell} = \mathcal T(x^{j_\ell})$, where $\mathcal T$ is the operator for the closed form solution to $(P_y)$ such that $y^{j_\ell}_{\Lcal} = \frac{ (x^{j_\ell})_{\Lcal} }{ \| (x^{j_\ell})_{\Lcal}  \|_2}$ and $y^{j_\ell}_{\Lcal^c}=0$. Letting $\mu_{j_\ell}:= \rho_{j_\ell} (\|x^{j_\ell}_{\Lcal}\|_2 - 1)$, it is easy to verify that $\rho_{j_\ell} (x^{j_\ell} - y^{j_\ell})_{\Lcal} = \mu_{j_\ell} y^{j_\ell}_{\Lcal}$ for all ${j_\ell}$. Hence, we write the above optimality condition as
\[
  0 =  M x^{j_\ell} -\lambda_{j_\ell} A x^{j_\ell} +  \begin{bmatrix} \rho_{j_\ell} (x^{j_\ell} - y^{j_\ell})_{\Lcal} \\ 0 \end{bmatrix} + \underbrace{\begin{bmatrix} 0 \\ \rho_{j_\ell} (x^{j_\ell} - y^{j_\ell})_{\Lcal^c}  \end{bmatrix} }_{:=w^{j_\ell}}
  =  M x^{j_\ell} -\lambda_{j_\ell} A x^{j_\ell} +  \mu_{j_\ell} \begin{bmatrix}  y^{j_\ell}_{\Lcal} \\ 0 \end{bmatrix} + w^{j_\ell},
\]
where $w^{j_\ell}_{\Lcal} =0$ for all ${j_\ell}$, and $0 \le \lambda_{j_\ell} \perp  (x^{j_\ell})^T A x^{j_\ell} - \phi \ge 0$.

We claim that the sequence $\big( (\lambda_{j_\ell}, \mu_{j_\ell}, w^{j_\ell}) \big)$ is bounded under the Robinson condition at $x^*$. Suppose not, consider the normalized sequence
\[
( \wt \lambda_{j_\ell}, \wt \mu_{j_\ell}, \wt w^{j_\ell}):= \frac{ (\lambda_{j_\ell}, \mu_{j_\ell}, w^{j_\ell})} { \| (\lambda_{j_\ell}, \mu_{j_\ell}, w^{j_\ell}) \|_2}, \qquad \forall \ {j_\ell}.
\]
 Then by working on a suitable convergent subsequence of $( \wt \lambda_{j_\ell}, \wt \mu_{j_\ell}, \wt w^{j_\ell})$ whose limit is given by $(\wt \lambda_*, \wt \mu_*, \wt w^*)$ with $\| (\wt \lambda_*, \wt \mu_*, \wt w^*) \|_2=1$, we obtain, in view of $x^*=y^*$, $y^*_{\Lcal^c}=0$ and the boundedness of $(M x^{j_\ell})$, and  passing the limits, that
\begin{equation} \label{eqn:limit_condition}
   -\wt\lambda_*  A x^* + \wt \mu_* x^* + \wt w^* =0,
\end{equation}
where $\wt \lambda_* \ge 0$, $x^*_{\Lcal^c} =0$, and $\wt w^*_\Lcal=0$. Consider (a): $(x^*)^T A x^* = \phi$, and (b): $(x^*)^T A x^* > \phi$ as follows.

(a) Consider $(x^*)^T A x^* = \phi$ first. By the Robinson's conditions at $x^*$ with the index set $\Ical=\Lcal$, we see that there exist a vector $d \in \mathbb R^N$ and a constant $v\in \mathbb R_-$ such that $-2 d^T A x^* - v = -2\wt \lambda_*$, $2 d^T x^* = - 2\wt \mu_*$, and $d_{\Lcal^c} = - \wt w^*_{\Lcal^c}$. Since $d_{\Lcal^c} = - \wt w^*_{\Lcal^c}$ and $\wt w^*_\Lcal=0$, we see that $d^T \wt w^* = -\| \wt w^* \|^2_2$. Therefore,
\[
  0 = -\wt \lambda_* d^T A x^* + \wt \mu d^T x^* + d^T \wt w^* =  - (\wt \lambda_*)^2 + \frac{\wt \lambda_* v }{2} - (\wt \mu_*)^2 - \|\wt w^* \|^2_2 = - \|  (\wt \lambda_*, \wt \mu_*, \wt w^*) \|^2_2 +  \frac{\wt \lambda_* v }{2}.
\]
This shows that $\|  (\wt \lambda_*, \wt \mu_*, \wt w^*) \|^2_2 =  \frac{\wt \lambda_* v }{2}$.
Since $\wt \lambda_* \ge 0$ and $v \le 0$, we have $\|  (\wt \lambda_*, \wt \mu_*, \wt w^*) \|^2_2=0$, contradiction. Therefore, the sequence $\big( (\lambda_{j_\ell}, \mu_{j_\ell}, w^{j_\ell}) \big)$ is bounded.

(b) Suppose $(x^*)^T A x^* > \phi$. In this case,  since $(x^{j_\ell})$ converges to $x^*$, we have  $(x^{j_\ell})^T A x^{j_\ell} > \phi$ for all $j_\ell$ sufficiently large. Hence, $\lambda_{j_\ell} =0$ for all large $j_\ell$. This shows that $\wt \lambda_*=0$.
By the Robinson condition at $x^*$, there exist a vector $d \in \mathbb R^N$ and a constant $v\in \mathbb R$ such that $-2 d^T A x^* - v = -2\wt \lambda_*$, $2 d^T x^* = - 2\wt \mu_*$, and $d_{\Lcal^c} = - \wt w^*_{\Lcal^c}$.  Applying these results to (\ref{eqn:limit_condition}) and using $\wt \lambda_*=0$, we obtain via a similar argument for case (a) that
$0 = -\wt \lambda_* d^T A x^* + \wt \mu d^T x^* + d^T \wt w^* =- \|  (\wt \lambda_*, \wt \mu_*, \wt w^*) \|^2_2 +  \frac{\wt \lambda_* v }{2} = \|  (\wt \lambda_*, \wt \mu_*, \wt w^*) \|^2_2$,  contradiction.
%
%

Consequently, the sequence $\big( (\lambda_{j_\ell}, \mu_{j_\ell}, w^{j_\ell}) \big)$ is bounded and has a convergent subsequence whose limit is given by $(\lambda_*, \mu_*, w^*)$ with $\lambda_* \ge 0$ and $w^*_\Lcal=0$. By passing the limit along this subsequence, we obtain
\[
  M x^* - \lambda_* A x^* + \mu_* x^* + w^* = 0, \quad 0\le \lambda_* \perp (x^*)^T A x^* - \phi \ge 0, \quad \|x^*\|_2=1, \quad w^*_{\Lcal}=0, \quad  x^*_{\Lcal^c}=0.
\]
%
%
Therefore, $x^*$, along with $(\lambda_*, \mu_*, w^*)$, is a KKT point for $(P)$ satisfying (\ref{eqn:KKT_conditions}) with $\Ical=\Lcal$.
\end{proof}

%
\section{Stage Two: Greedy Algorithm for Improvement} \label{sec: greedy_stage}

A greedy algorithm is proposed to improve  the result obtained from the penalty decomposition scheme (i.e., Algorithm~\ref{algo:PD_scheme}) in the second stage.
%
%
To describe this greedy algorithm, we introduce the following problem. Given an index set $\Ical \subset\{1, \ldots, N\}$, consider the restricted QCQP problem of (\ref{pr: p}):
\begin{equation} \label{eqn:P_I}
(P_{\Ical}): \quad  \min_{x \in \mathbb{R}^N} \ x^T M  x\qquad \mbox{subject to}  \qquad  x^T A x\ge \phi, \quad  x^T x=1, \quad \text{and}\quad x_{\Ical^c}=0.
\end{equation}
Using this subproblem, the proposed greedy algorithm is presented as follows.

\begin{algorithm}[H]
\caption{Greedy Algorithm}
\begin{algorithmic}[1]
\label{algo:greedy_scheme}
\STATE Input: $M \text{ and } A \in \mathbb R^{N\times N}$, $\phi >0$, $k\in \mathbb{N}$, $s=2$, and the index set $\Lcal$ with $|\Lcal|=k$ obtained from Algorithm~\ref{algo:PD_scheme}

\STATE Initialization: set $n=0$, and $S^{(n)}=\Lcal$

\REPEAT

\STATE  Find $\Jcal^{(n)} \in \Argmin_{\Jcal\subseteq [S^{(n)} ]^c, \ |\Jcal|=s}$ $(P_{S^{(n)}\cup \Jcal})$, where each $(P_{S^{(n)}\cup \Jcal})$ is solved  via Algorithm~\ref{algo:solve_P_I}

%
\STATE Find $S^{(n+1)} \in \Argmin_{\mathcal K \subset [ S^{(n)} \cup \Jcal^{(n)}], \ |\mathcal K|=k } (P_{\mathcal K})$, where $(P_{\mathcal K})$ is solved via Algorithm~\ref{algo:solve_P_I}

\STATE $n \leftarrow n+1$

\UNTIL $S^{(n)}=S^{(n-1)}$

\STATE Output: an optimal solution $x^*$ to $(P_{  S^{(n)} })$.


\end{algorithmic}
\end{algorithm}

The proposed greedy algorithm improves an underlying index set $S^{(n)}$ and its associated solution $x^{(n)}$ in each iteration. We describe its key ideas and steps as follows. For each $n$th iterate, the algorithm first selects two additional best local indices by solving ${N-k \choose 2}$ restricted QCQPs on  $\mathbb R^{(k+2)}$ (cf. Line 4). It then chooses an optimal index subset $S^{(n+1)}$ of size $k$ from the union of $S^{(n)}$ and the two additional indices which achieves a minimal value among all the $\frac{k(k+1)}{2}$ restricted QCQPs on  $\mathbb R^k$.
The greedy algorithm improves the solution of the PD scheme as shown below.

\begin{lemma}
 Let $f(x) = x^T M x$ be the objective function of $(P)$ given in (\ref{pr: p}), and  $\big(x^{(n)}\big)$ be a sequence generated by Algorithm~\ref{algo:greedy_scheme}, where $x^{(n)}$ is the solution of $P_{S^{(n)}}$. Then each $x^{(n)}$ is feasible to $(P)$, and $\big ( f( x^{(n)} ) \big)$ is decreasing.
\end{lemma}

\begin{proof}
 Note that for each $n$, $|S^{(n)}|=k$ and $ x^{(n)}_i=0$ for all $i \notin  S^{(n)}$. Clearly, $x^{(n)}$ is a feasible point of $(P)$. Furthermore, for the index set $\Jcal^{(n)}$ obtained from Line 4, $S^{(n)}$ is a subset of $S^{(n)} \cup \Jcal^{(n)}$ with $|S^{(n)}|=k$. By  Line 5 of Algorithm~\ref{algo:greedy_scheme}, we see that $f( x^{(n+1)}) \le f( x^{(n)})$ and $\big ( f( x^{(n)} ) \big)$ is decreasing.
\end{proof}


We discuss an algorithm solving the subproblem (\ref{eqn:P_I}) for a given index set $\Ical$ with $|\Ical|=\ell$. Define
\begin{equation} \label{eq: q0_q1}
    Q_0 := M_{\Ical \Ical} \quad \text{and} \quad Q_1 :=-\frac{A_{\Ical \Ical}}{\phi}.
\end{equation}
Since  $A$ and $M$ are both positive definite,  $Q_0\succ 0\succ Q_1$. Clearly, $(P_\Ical)$ can be reduced to the following nonconvex QCQP:
\begin{equation} \label{pr: upsilon_step_update}
 (P_\Upsilon): \qquad \min_{\Upsilon \in \mathbb{R}^{|\Ical|} } \ \Upsilon^TQ_0\Upsilon \qquad \mbox{subject to}  \qquad  \Upsilon^TQ_1\Upsilon\le  -1, \quad \text{and}\quad  \Upsilon^T\Upsilon=1.
\end{equation}
To solve $(P_{\Upsilon})$, consider its    SDP  relaxation:
\begin{equation}  \label{pr: sdp-relax-py}
(P_Y):  \qquad \min_{Y\in \mathcal S^{|\Ical|}}  \quad
\tr\left(Q_0Y\right)             \qquad
\textrm{subject to} \qquad
 \tr\left(Q_1Y\right)\le -1,\quad  \tr(Y)=1, \quad \text{and} \quad Y\succeq 0.
\end{equation}
The reason for considering  (\ref{pr: sdp-relax-py}) instead of (\ref{pr: upsilon_step_update}) is threefold. First, it is known that the nonconvex problem (\ref{pr: upsilon_step_update}) and its convex  SDP  relaxation (\ref{pr: sdp-relax-py}) have the same optimal value. Second, given a solution of (\ref{pr: sdp-relax-py}), a rank-one  solution of (\ref{pr: sdp-relax-py}), which  can be constructed in a polynomial time, leading to a global  solution of  (\ref{pr: upsilon_step_update}). Lastly, since $|\Ical|$ is usually small, leveraging   SDP  relaxation is computationally efficient.

Let $Y^*$ be a solution of (\ref{pr: sdp-relax-py}). Consider two cases:  $\tr(Q_1Y^*)<-1$, and $\tr(Q_1Y^*)=1$. For the former case, it is easy to see that the optimal value is $\lambda_{\min}(Q_0)$, and $Y^*=vv^T$, where $v$ is the unit eigenvector associated with $\lambda_{\min}(Q_0)$.
For the latter case, the problem becomes
\begin{equation}  \label{pr: sdp-relax-py-equality}
\min_{Y\in \mathcal S^{|\Ical|}}  \quad
\tr\left(Q_0Y\right)             \qquad
\textrm{subject to} \qquad
 \tr\left(Q_1Y\right)= -1,\quad  \tr(Y)=1, \quad \text{and} \quad Y\succeq 0.
\end{equation}
Its dual problem is:
\begin{equation*}  \label{pr: dual-sdp-relax-py}
\max_{y=(y_1, y_2)\in \mathbb R^2} \quad -y_1+y_2
   \qquad
\textrm{subject to} \qquad
Z=Q_0-y_1Q_1-y_2I_k,\quad \text{and} \quad Z\succeq 0,
\end{equation*}
which clearly has a strictly feasible point because $Q_0\succ 0$. Hence, the dual problem satisfies the Slater's condition such that the strong duality holds, i.e., the primal and dual problems have the same optimal value.
%
%
Furthermore,
 it is known that the standard SDP with $m$ constraints has  a solution with rank $r$ such that $r(r+1)\le 2 m$ \cite{barvinok1995problems,pataki1998rank}. We immediately deduce  that (\ref{pr: sdp-relax-py-equality}), if feasible, has a rank-one solution. 

There are various standard methods to construct a solution of a QCQP from its  SDP  relaxation in a polynomial time
 \cite{ai2011new,ai2009strong, sturm2003cones}. However,  we avoid describing  them here and instead  discuss a rank reduction  algorithm for this  SDP \cite{lemon2016low}, which is more appropriate for practical purposes. The following algorithm starts from an arbitrary solution of (\ref{pr: sdp-relax-py-equality}) and returns another solution with rank one.

\begin{algorithm}[H]
\caption{Solution Rank Reduction for (\ref{pr: sdp-relax-py-equality})}
\begin{algorithmic}[1]
\label{algo: rank-reduction}
%
\STATE Input: A solution $Y$ of (\ref{pr: sdp-relax-py-equality}). Let $r=\rank(Y)$.

\WHILE{$r>1$}
\STATE
Compute a factorization $Y=VV^T$ with $V\in \mathbb R^{N\times r}$
\STATE Find a nonzero $\Delta \in \mathcal S^r$ with $\tr(\Delta)=0$ and  $\tr(V^TQ_1V\Delta)=0$
\STATE Let $Y=V
\left(I_r-{\lambda^{-1}_{\max}(\Delta)}\Delta\right)V^T$
\STATE Let $r=\rank(Y)$
\ENDWHILE

\STATE Output: A rank-one solution of (\ref{pr: sdp-relax-py-equality}).

\end{algorithmic}
\end{algorithm}

 We are now ready to  summarize our procedure for solving (\ref{eqn:P_I}) in the following algorithm.

\begin{algorithm}[H]
\caption{Solve Problem~$(P_\Ical)$ given by (\ref{eqn:P_I})}
\begin{algorithmic}[1]
\label{algo:solve_P_I}
%
\STATE Input: $Q_0 \text{ and } Q_1$ defined in (\ref{eq: q0_q1}).

\STATE
Compute $\lambda_{\min}(Q_0)$ and its associated unit eigenvector $v$.
\STATE Solve (\ref{pr: sdp-relax-py-equality}) to get a solution $X$.
\STATE Apply Algorithm \ref{algo: rank-reduction} to obtain a rank-one solution $X=uu^T$ of (\ref{pr: sdp-relax-py-equality})
\IF{  $v^T Q_1v<1$ and $\lambda_{\min}(Q_0)< u^TQ_0u$ }
\STATE $\Upsilon=v$
\ELSE
\STATE $ \Upsilon=u$
\ENDIF

\STATE Output: $x_\Ical=\Upsilon$, and $x_{\Ical^c}=0$

\end{algorithmic}
\end{algorithm}

%

\section{Numerical Experiments} \label{sec: numerical_experiments}

A standard (statistical) arbitrage strategy has the following four main steps: constructing an initial appropriate asset pool, designing a mean-reverting portfolio, applying a unit-root test to verify the mean-reverting property, and finally, trading such a favorable portfolio.  For the first step, a method for constructing an asset pool is explained in \cite{cuturi2016mean}, which resorts to the smallest eigenvalue of the corresponding covariance matrix.  In this paper, we focus on the very critical (second) step of an arbitrage strategy, namely, proposing a mean-reverting portfolio through minimizing the predictability surrogate. Additionally, we include the sparsity and volatility constraints in (\ref{pr: p}) to ensure the profitability of such a portfolio; when utilized in practice. For the third step, it is well-known how to employ a unit-root test such as Dickey-Fuller \cite{dickey1979distribution}. For the forth step, we avoid describing a regular mean-reversion trading strategy and instead, we refer the readers to \cite{zhao2018mean,zhao2019optimal,zhao2018optimal} for a comprehensive discussion on how to benefit from such a portfolio.  However, we briefly explain about three standard  performance metrics, often used in the literature, for the numerical experiments or portfolio investment \cite{zhao2018mean, zhao2019optimal}.

\gap

\noindent $\bullet$ Cumulative profit and loss ($P\&L$): this tool measures the cumulative return of a mean-reverting portfolio in one trading period from $t_1$ to $t_2$ based on the following formula:
\begin{equation*}
  \text{Cum.  P\&L} \,  (t_1,t_2)= \sum_{t=t_1}^{t_2} \text{P\&L}_t,
\end{equation*}
where $\text{P\&L}_t= y^Tr_t(t-t_o)-y^Tr_{t-1}(t-1-t_o)$ whenever   a long position is opened  and  $\text{P\&L}_t= y^Tr_t(t-t_o)-y^Tr_{t-1}(t-1-t_o)$ provided that a short position is opened at time $t_o$.
 For a given asset, we have
\begin{equation*}
    r_t(\tau)=\frac{p_t-p_{t-\tau}}{p_{t-\tau}}\approx \ln (p_t)-\ln (p_{t-\tau}),
\end{equation*}
and $p_t$ denotes the price of a considered asset at time $t$.   We use Table 1 in \cite{zhao2018mean}  with $d$ equal to the suggested standard deviation of the portfolio.

\gap

\noindent $\bullet$ Sharpe Ratio: this is another metric to examine the quality of a portfolio for a period of time from $t_1$ to $t_2$ is defined as $SR_{ROI}(t_1,t_2)=\mu_{ROI}/\sigma_{ROI},$ where $\mu_{ROI}=1/(t_2-t_1)\sum_{t=t_1}^{t_2} ROI_t$ and $\sigma^2_{ROI}=1/(t_2-t_1)\sum_{t=t_1}^{t_2} (ROI_t-\mu_{ROI})^2$. Clearly, a  portfolio with a bigger Sharpe ratio is more profitable.

\gap

\noindent
$\bullet$ Return on investment (ROI): this tool measures the investment return of a mean-reverting portfolio  is defined as follows:
\begin{equation*}
    \text{ROI}_t=\frac{\text{P\&L}_t}{\|y\|_1}.
\end{equation*}

The setup of the numerical experiments is as follows. We first select real data coming from  the U.S. stock market, i.e., the  Standard and Poor's 500 (S\&P 500) Index, which is often used in the state-of-the-art literature. Here, we first combine the asset pools suggested in \cite{cuturi2016mean,zhao2018mean,zhao2018optimal, zhang2020sparse} and then, add several assets to have an asset pool with 30 assets. The trading time period is considered from February 1st, 2012 to June 30th, 2014.
The log-prices of these assets are depicted in Figures \ref{fig1}, \ref{fig2} and \ref{fig3}.
The volatility threshold $\phi$ is chosen based on the idea proposed in \cite{cuturi2016mean}, namely,  we choose this parameter  to be larger than one fifth of the median variance of all  assets in the pool. Four different sparsity levels, i.e., $k=4,5,6,7$, are considered.
To the best of our knowledge,  the only existing method for (indirectly) tackling (\ref{pr: p}) is the SDP relaxation method in \cite{cuturi2016mean}. Thus, we only compare the proposed scheme denoted by PD-G with this method. Recall that the SDP relaxation method first solves (\ref{pr: sdp-of-p}) to get $Y^*$ and then solves (\ref{pr: sparse-pca}) for $H=Y^*$. The sparse PCA problem  (\ref{pr: sparse-pca}) is solved via the method given in \cite{luss2013conditional}. For the proposed PD-G scheme, we use the following stopping criteria:
\begin{equation*} 
\max
\left\{
\frac{\|x^{j,s}-x^{j,s-1}-\|_\infty}{\max \left(\|x^{j,s}\|_\infty,1 \right)},
\frac{\|y^{j,s}-y^{j,s}\|_\infty}{\max \left(\|y^{j,s}\|_\infty,1 \right)}
\right\}
\le
5\times 10^{-3},
\end{equation*}
 $\max
\left\{
\|x^j-y^j\|_{\infty}
\right\}
\le
5\times 10^{-4} $,
and $\rho_{j}=\sqrt{10}\rho_{j-1}$. 

%
%
%


\begin{figure}[H]
\centering
 \includegraphics[width=0.8\textwidth]{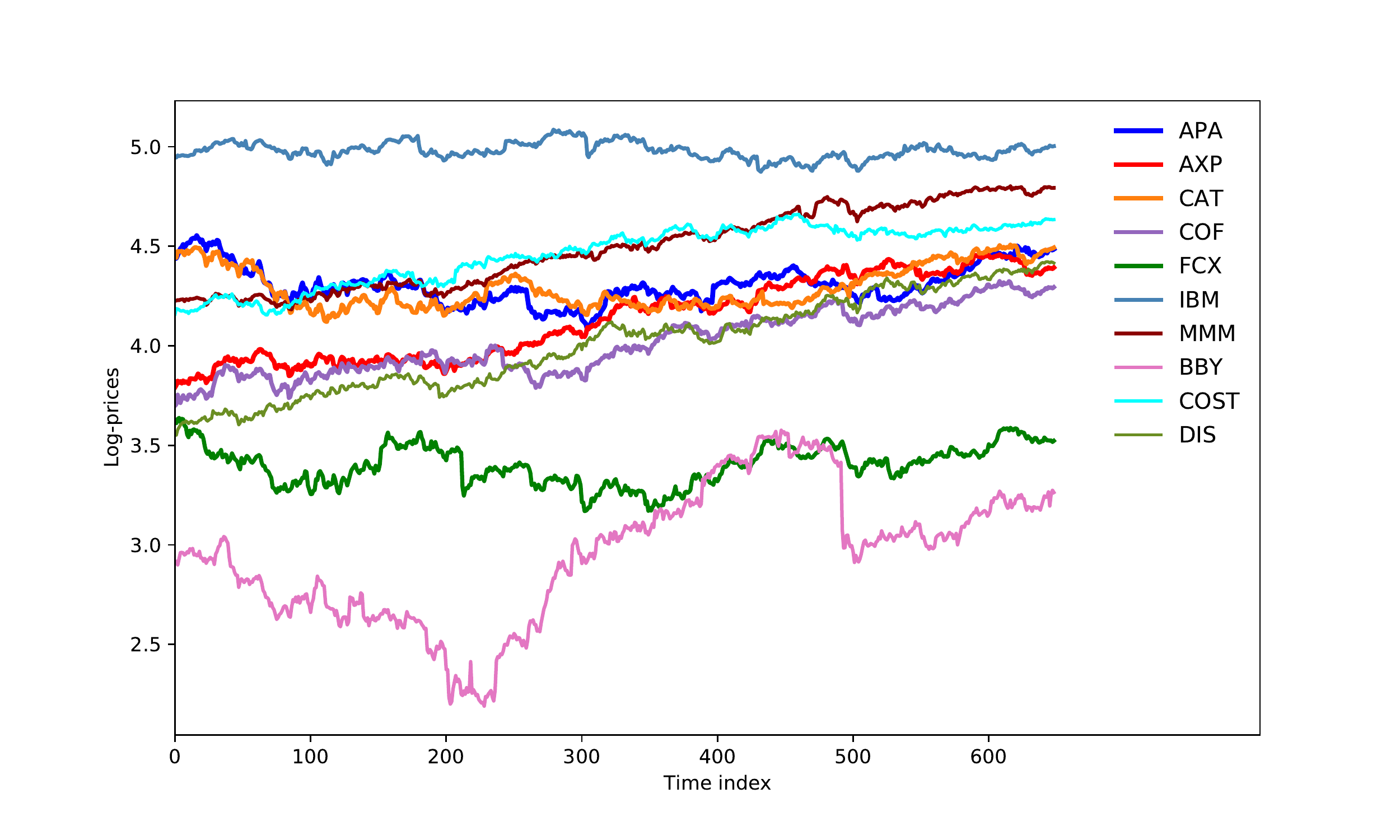}
  \caption{Log-prices for the first 10 assets.} \label{fig1}
  \end{figure}

\begin{figure}[H]
\centering
 \includegraphics[width=0.8\textwidth]{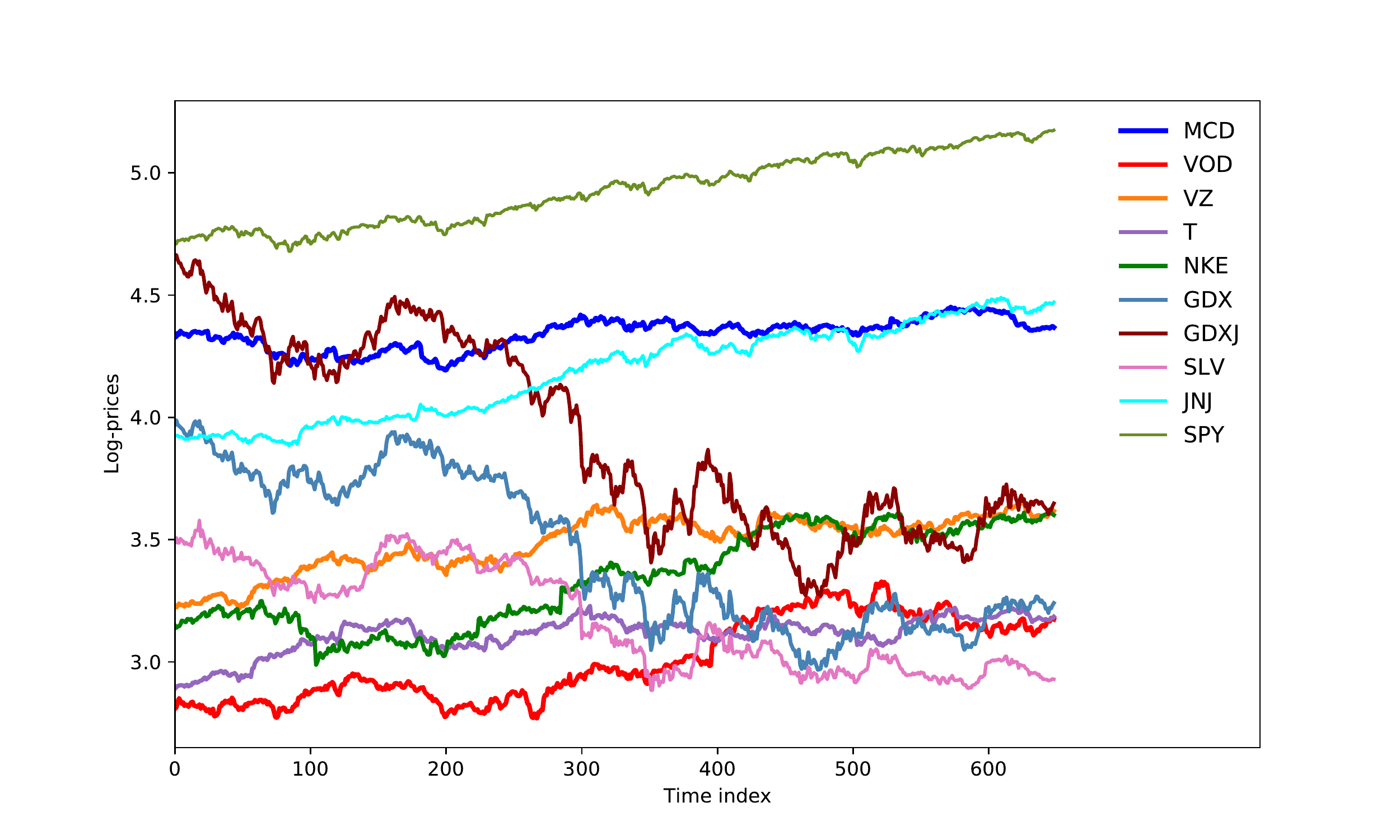}
  \caption{Log-prices for the second 10 assets.}  \label{fig2}
  \end{figure}

\begin{figure}[H]
\centering
 \includegraphics[width=0.8\textwidth]{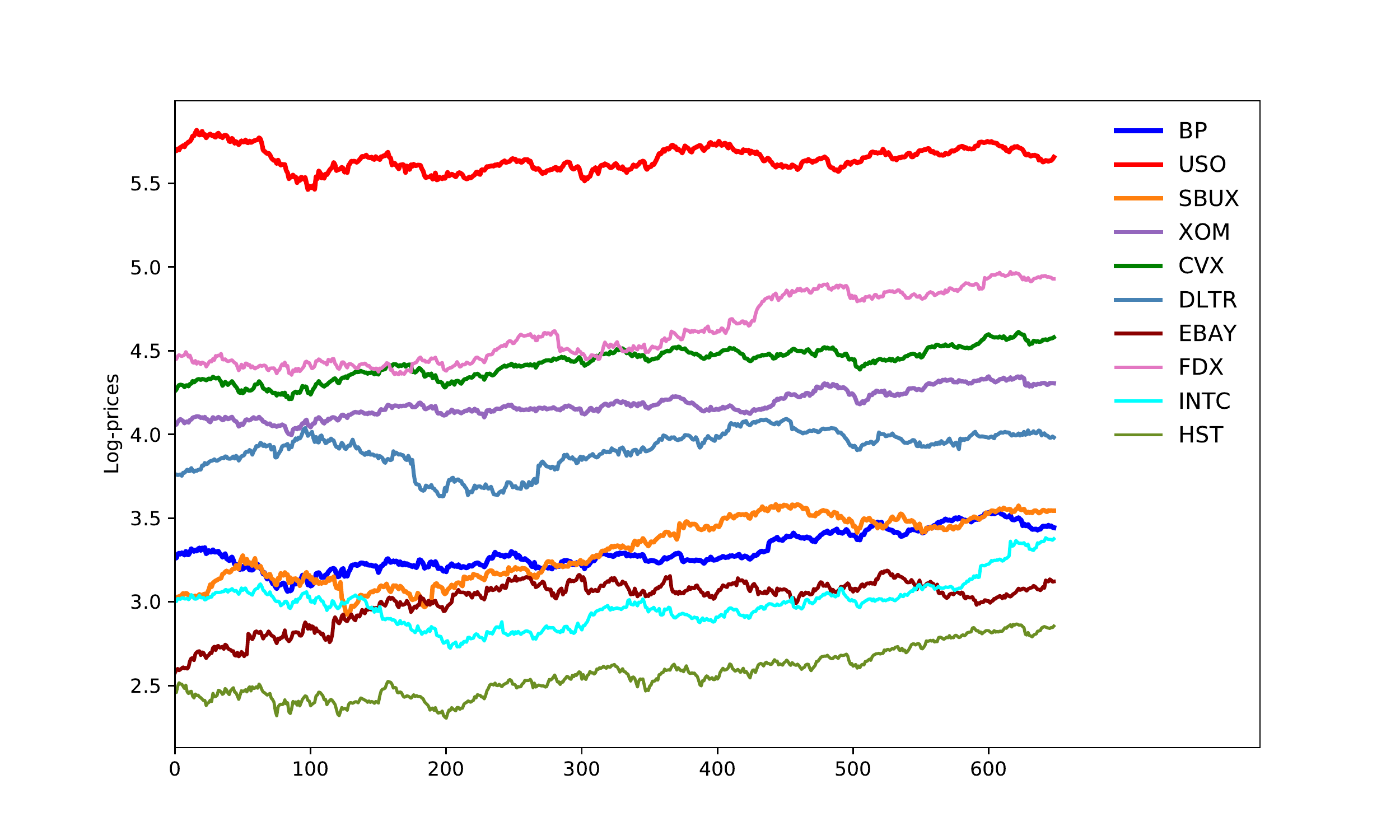}
  \caption{Log-prices for the third 10 assets.} \label{fig3}
  \end{figure}

\begin{figure}[H]
\centering
  \includegraphics[width=0.9\textwidth]{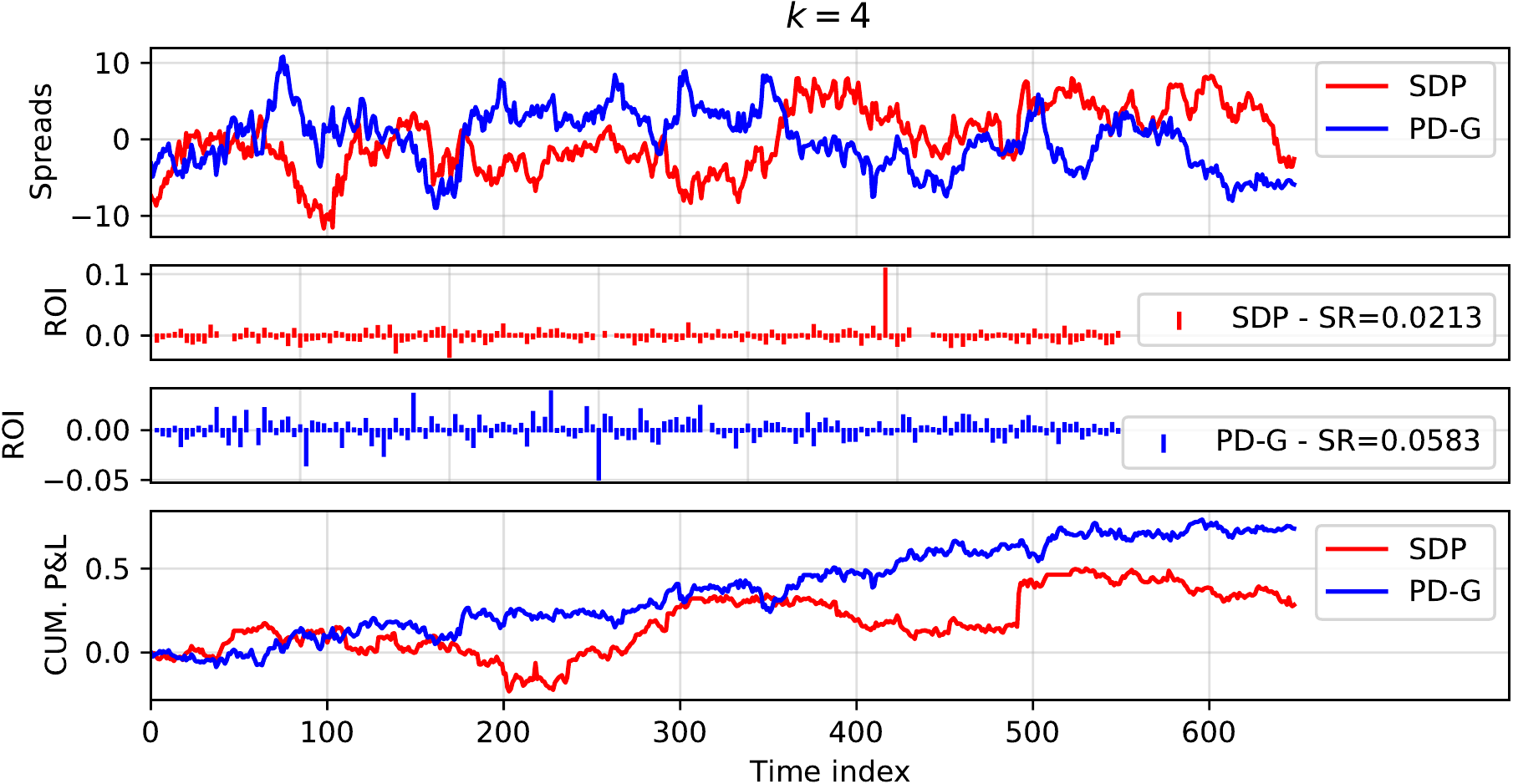}
  \caption{Time history for $k=4$; Optimal value: SDP=16.37 and PD-G=7.01} \label{fig:k=4_time}
\end{figure}

\begin{figure}[H]
\centering
 \includegraphics[width=0.85\linewidth]{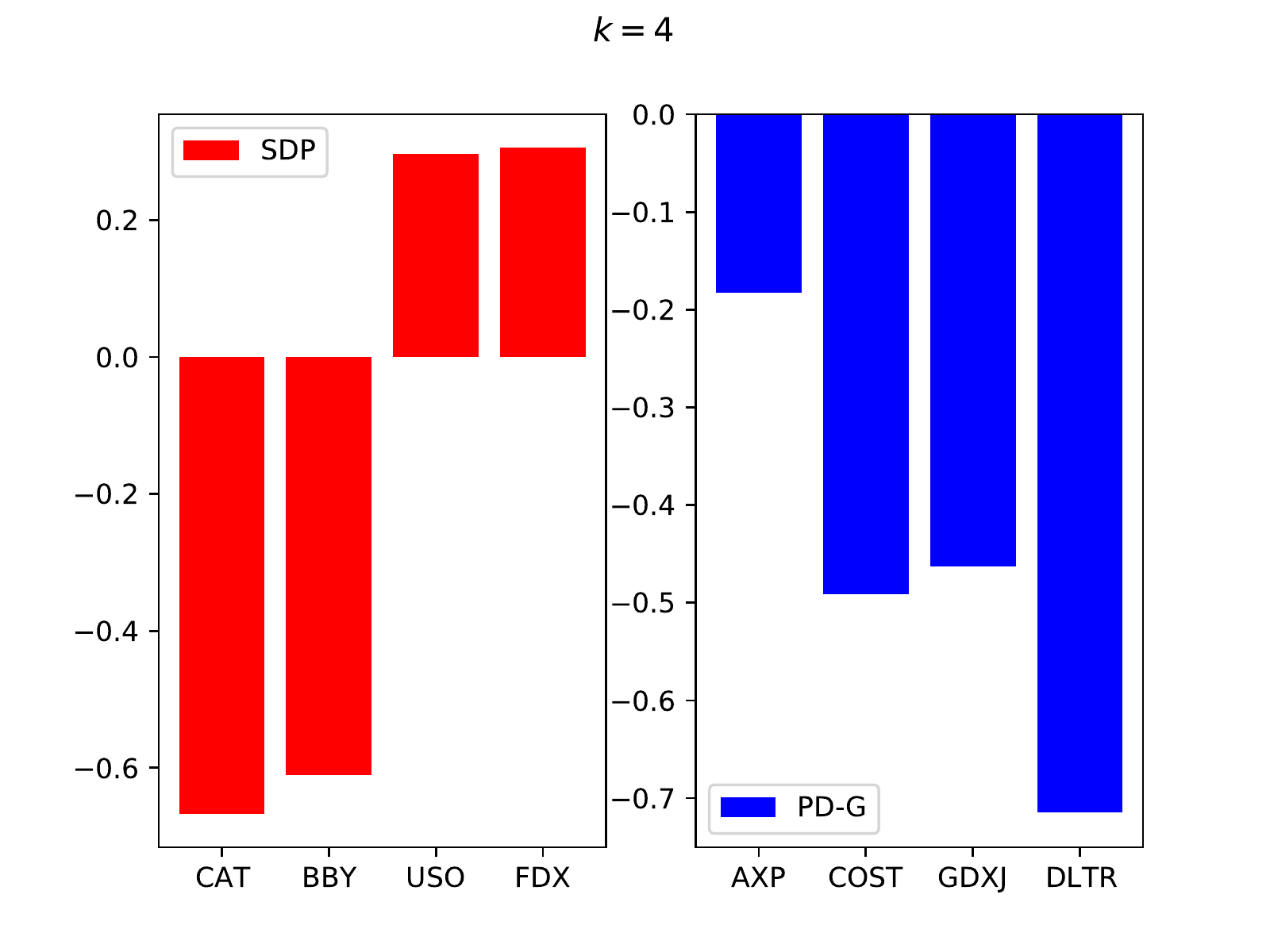}
 \caption{Portfolios selected  by SDP and PD-G for $k=4$ } \label{fig:k=4_portfolio}
  \end{figure}


\begin{figure}[H]
\centering
  \includegraphics[width=0.9\textwidth]{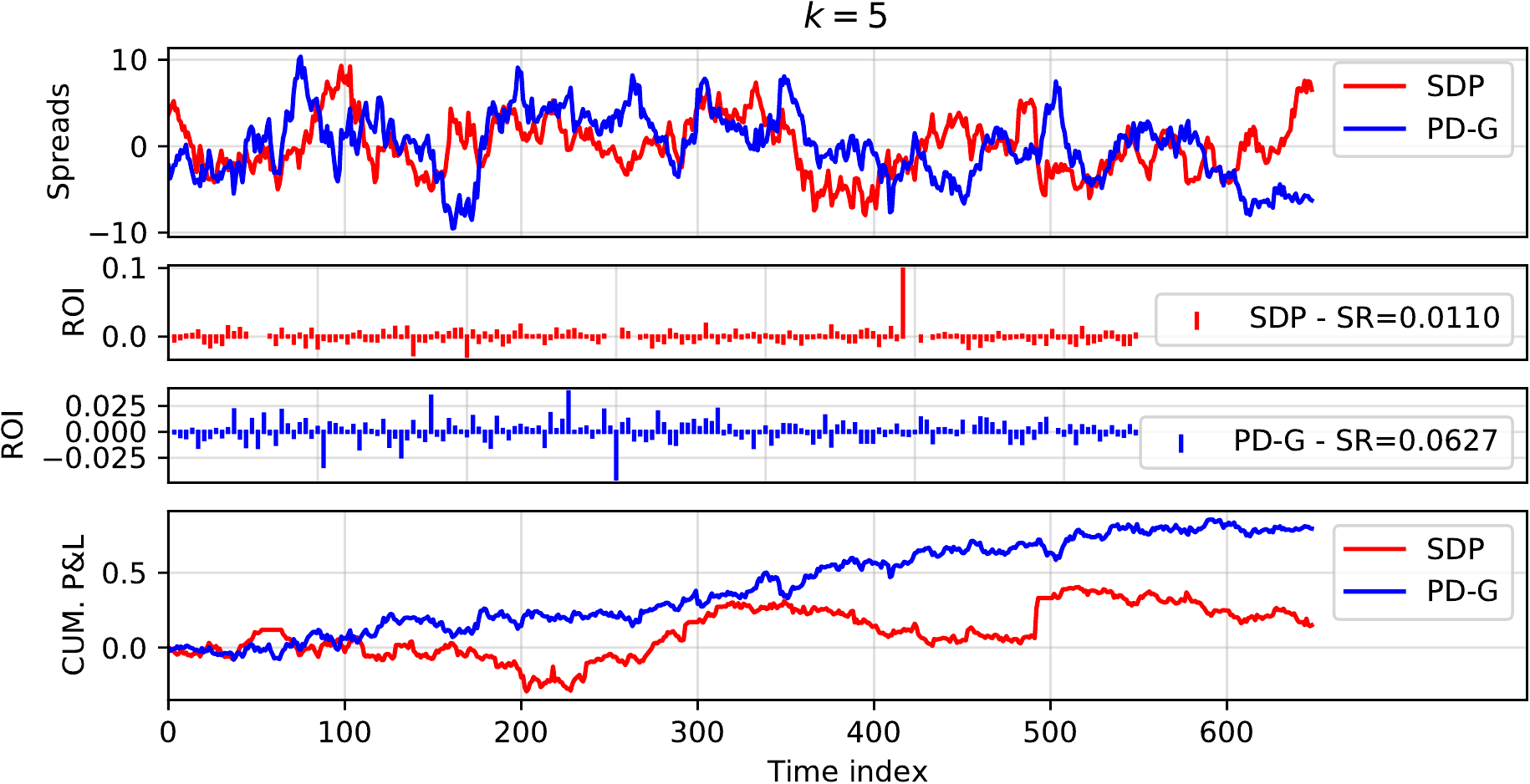}
  \caption{Time history for $k=5$; Optimal value: SDP=16.17 and PD-G=6.91.} \label{fig:k=5_time}
\end{figure}

\begin{figure}[H]
\centering
 \includegraphics[width=0.85\linewidth]{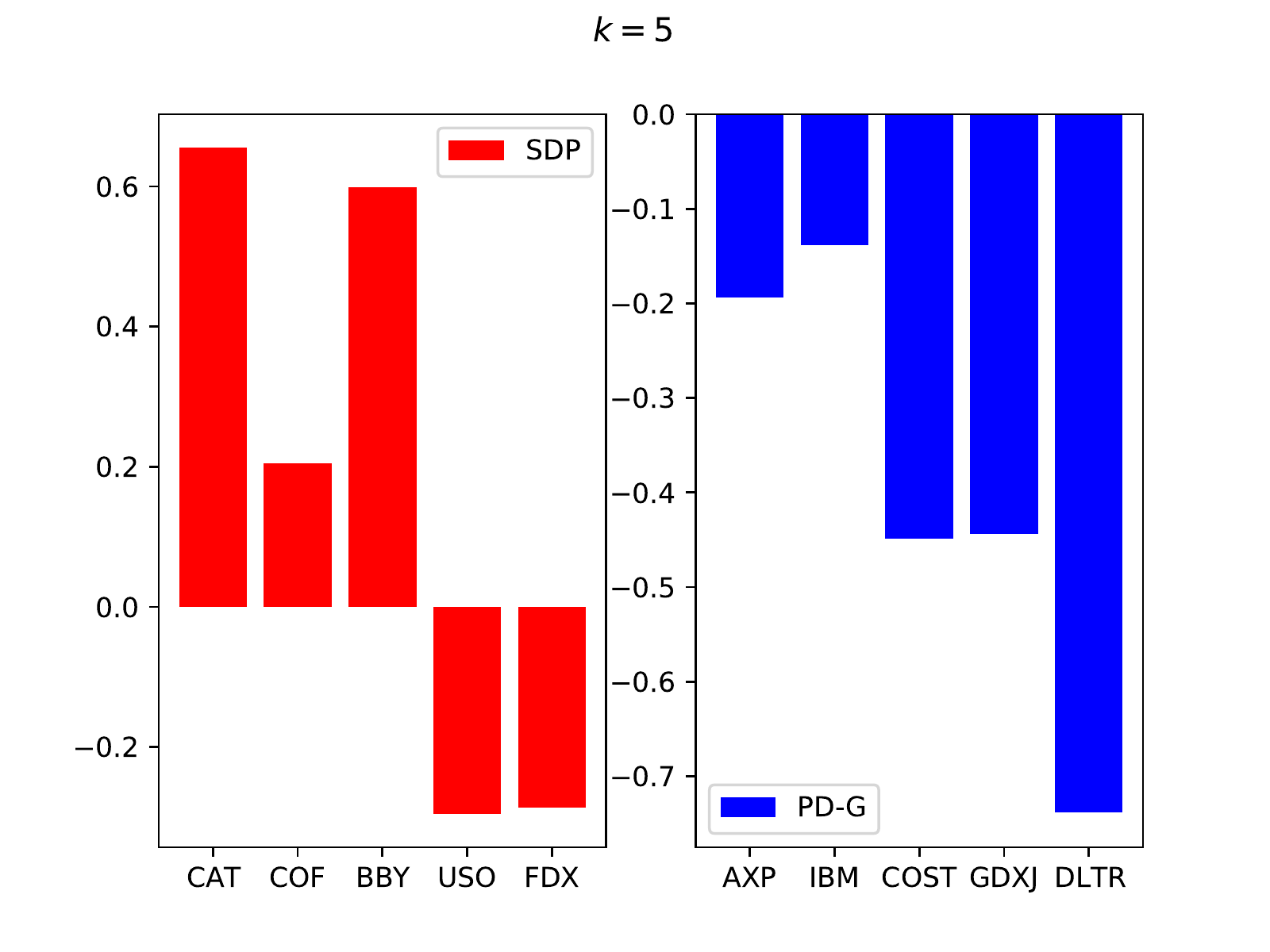}
 \caption{Portfolios selected  by SDP and PD-G for $k=5$ } \label{fig:k=5_portfolio}
  \end{figure}

\begin{figure}[H]
\centering
  \includegraphics[width=0.9\textwidth]{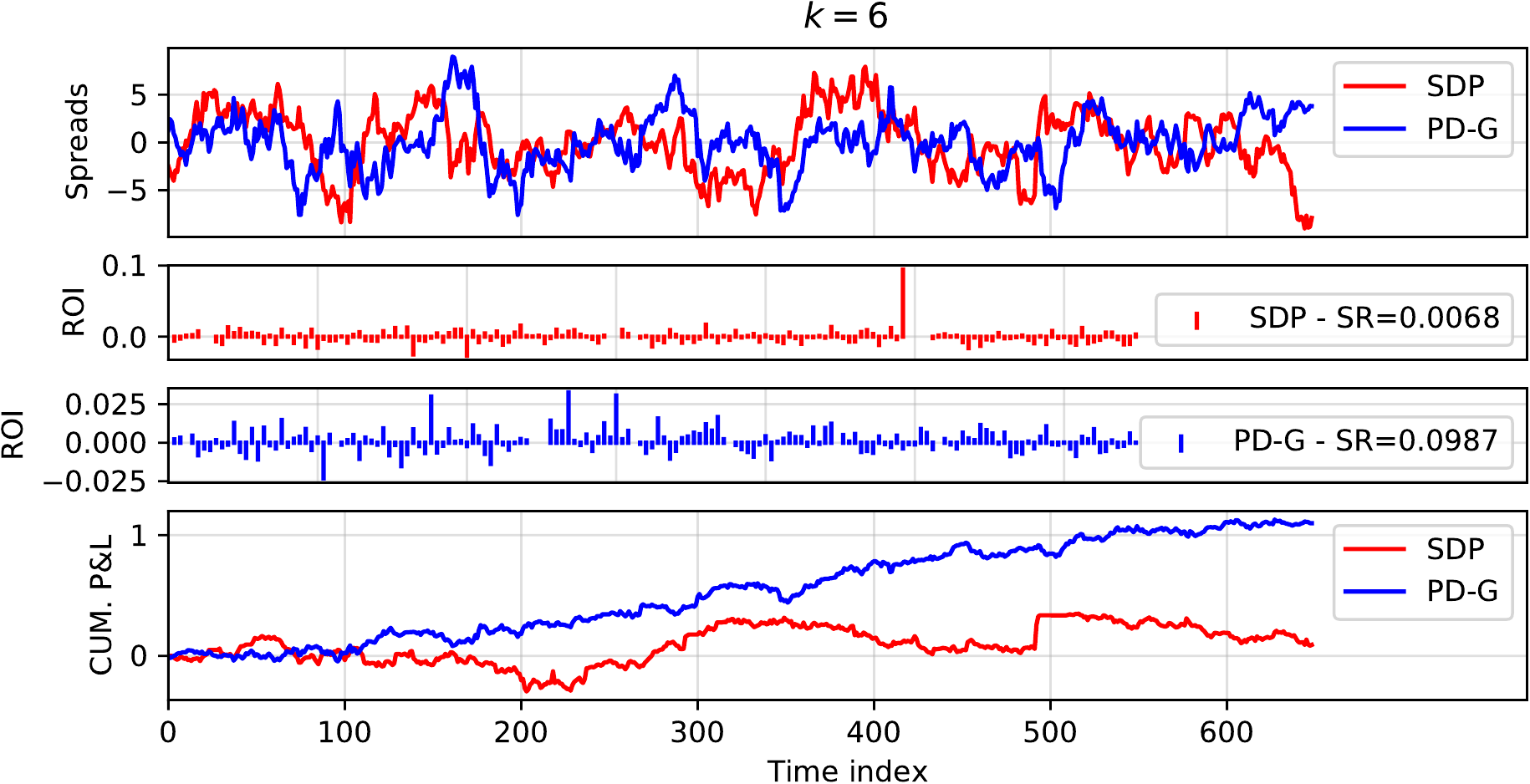}
  \caption{Time history for $k=6$; Optimal value: SDP=9.34 and PD-G=6.77.} \label{fig:k=6_time}
\end{figure}

\begin{figure}[H]
\centering
 \includegraphics[width=0.85\linewidth]{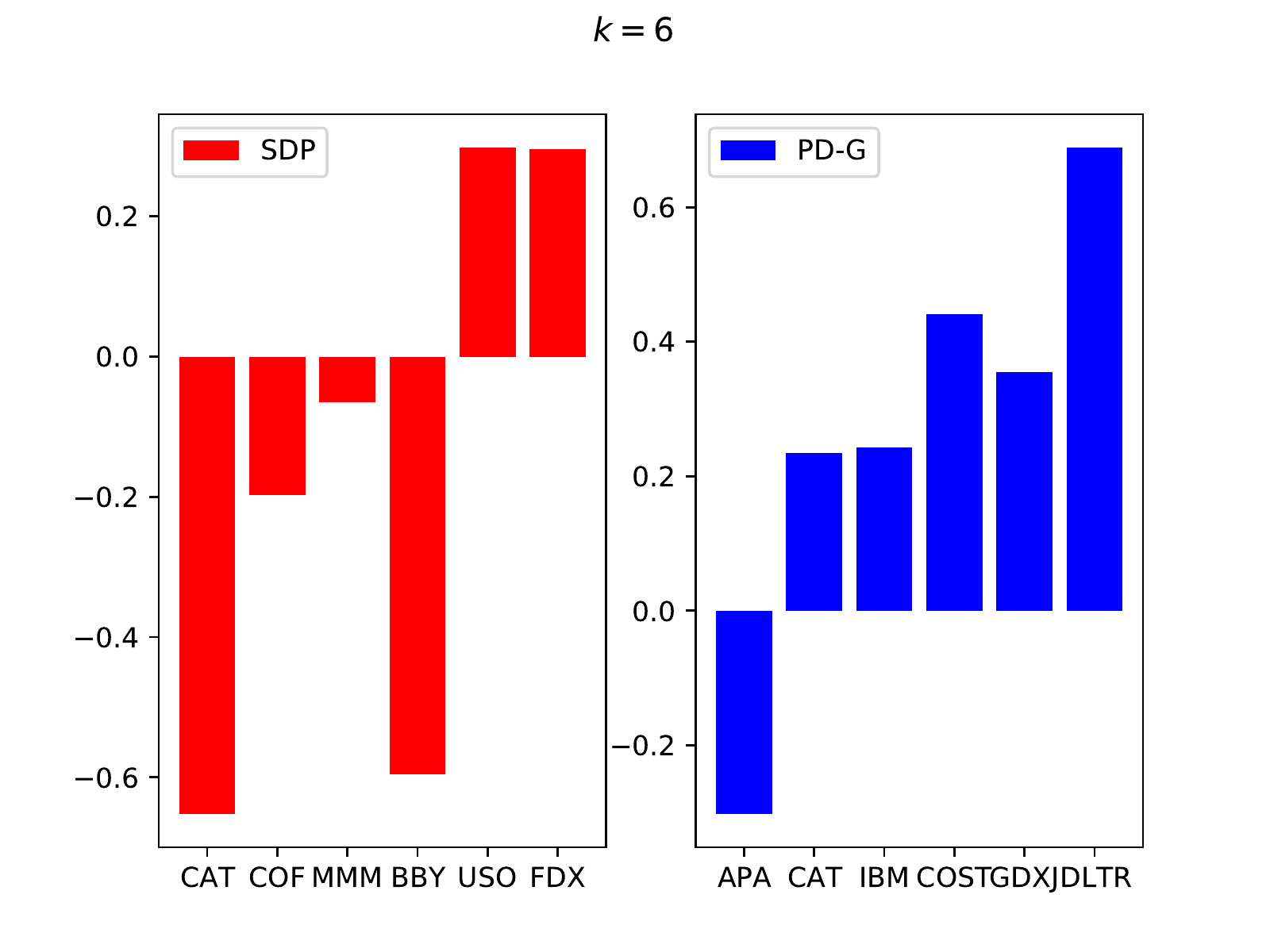}
 \caption{Portfolios selected  by SDP and PD-G for $k=6$ }  \label{fig:k=6_portfolio}
  \end{figure}

\begin{figure}[H]
\centering
  \includegraphics[width=0.9\textwidth]{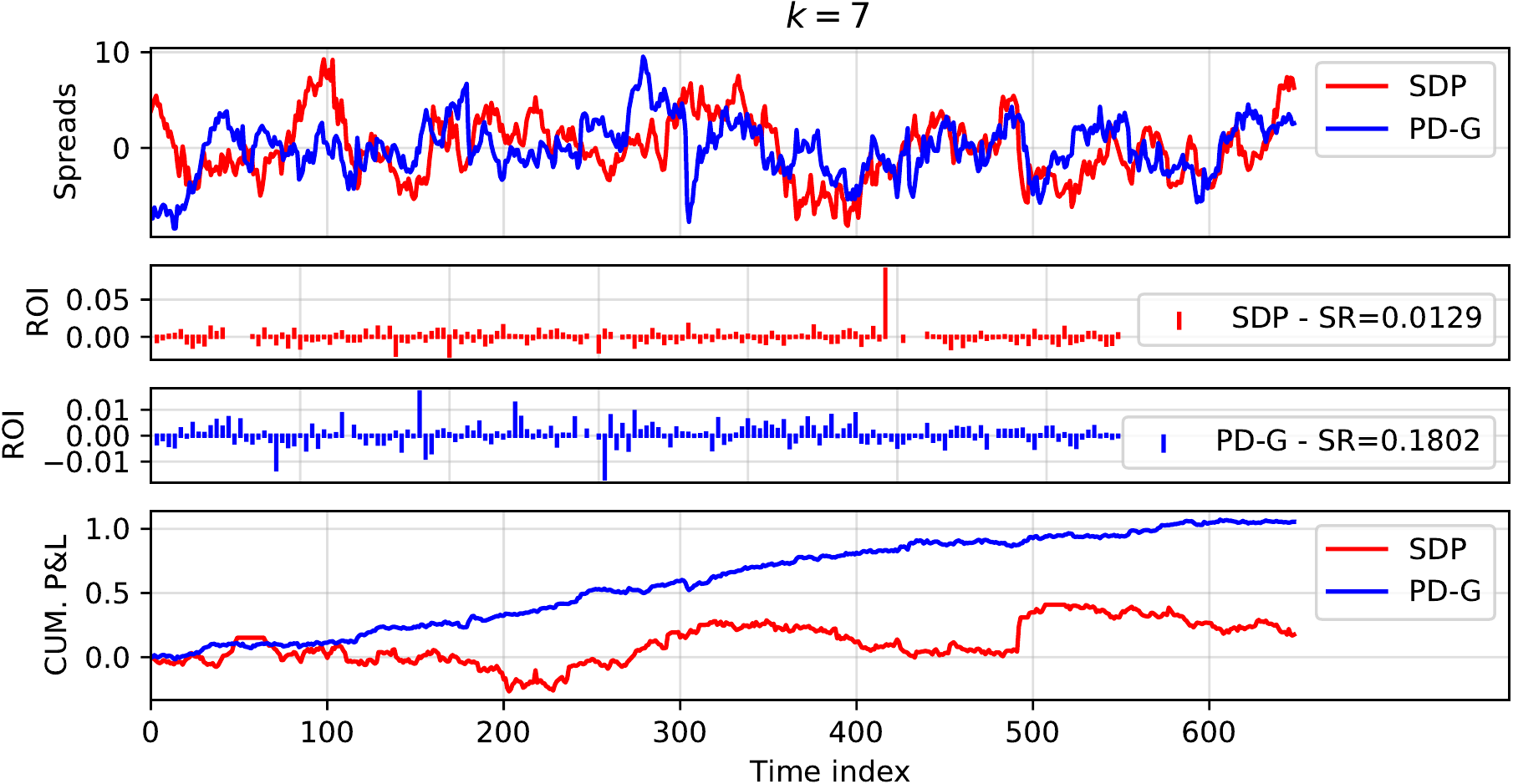}
  \caption{Time history for $k=7$; Optimal value: SDP=9.33 and PD-G=6.66.} \label{fig:k=7_time}
\end{figure}

\begin{figure}[H]
\centering
 \includegraphics[width=0.85\linewidth]{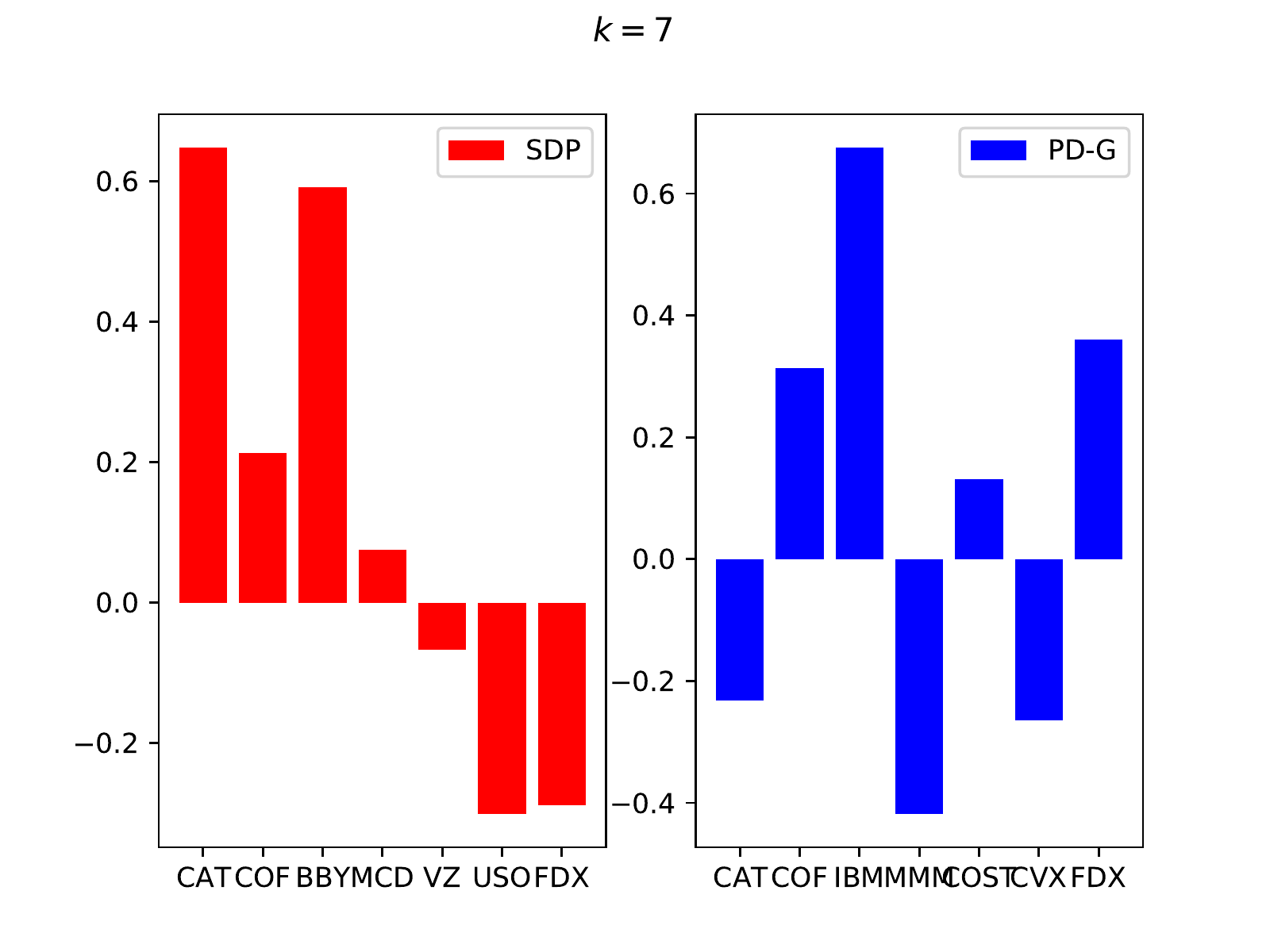}
 \caption{Portfolios selected  by SDP and PD-G for $k=7$ }  \label{fig:k=7_portfolio}
  \end{figure}

We compare the profitability of  the proposed PD-G scheme with the SDP relaxation method for each sparsity level $k$. The numerical results are displayed in Figures~\ref{fig:k=4_time}-\ref{fig:k=7_portfolio} for different $k$'s, where both the time history of spreads, return on investment (ROI), and cumulative profit and loss (Cum. P\&L) and the portfolios selected by the two methods are shown for each $k$. 
It is seen that for each $k$, the spread produced by the PD-G method captures the desired mean-reverting property very well, while the spread from the SDP relaxation method does not always demonstrate this property over the whole period in consideration. It is also seen that the portfolios chosen by the PD-G method achieve better Sharpe ratios compared to those from the SDP method. More precisely, the Sharpe ratio from the PD-G method is twice larger for $k=4$, six times larger for $k=5$, much larger for $k=6$, and twice bigger for $k=7$ to that of the SDP method, respectively. Further, the figures of cumulative returns for the proposed PD-G method show a similar increasing trend for each  sparsity level whereas  the inconsistent trends are shown from the SDP method. Consequently, the final cumulative returns generated by the PD-G method are  bigger  for each $k$. Besides, the  optimal values obtained from the two methods are quite different. In particular, the PD-G method achieves much smaller optimal values than the SDP method for all the four sparsity levels. Hence, the proposed PD-G method not only attains favorable theoretical properties but also outperforms the SDP relaxation scheme on the benchmark data set.

%

\section{Conclusion}

This paper proposes a two-stage algorithm to solve the mean-reverting portfolio optimization problem subject to sparsity and volatility constraints. In the first stage, a penalty decomposition scheme is used, and in the second, a SDP-relaxation based greedy algorithm is invoked. Theoretical properties of this algorithm are established, and numerical results demonstrate the efficacy of the proposed scheme.

%
\section*{Acknowledgement}

We greatly appreciate Dr. Zhaosong Lu of University of Minnesota for bringing our attention to the penalty decomposition method with improvement via a greedy algorithm for the mean-reverting problem and many helpful discussions on the proposed scheme and its implementation.
We also very much appreciate Dr. Shuzhong Zhang of University of Minnesota  for several prolific discussions on  the SDP relaxation topic and further,  providing us with many helpful references. Lastly, we deeply thank Mr. Hartmut Durchschlag of Cargill for many insightful inputs on mean-reverting portfolios.



{\small
\bibliographystyle{abbrv}
\bibliography{main}
}

\end{document}